\documentclass{elsarticle}
\usepackage[utf8]{inputenc}[11pt]
\usepackage{geometry}
\usepackage[official]{eurosym}
\usepackage{amsmath,amssymb}
\usepackage{graphicx}
\usepackage{lineno}
\usepackage{multirow}
\usepackage{lscape}
\usepackage{diagbox}

\newtheorem{proposition}{Proposition} 
\newtheorem{lemma}{Lemma} 
\newenvironment{proof}{{\noindent\it Proof}\quad}{\hfill $\square$\par}

\usepackage{natbib}
\usepackage{graphicx}
\usepackage{algorithm}  
\usepackage{algorithmicx}  
\usepackage{algpseudocode}  
\usepackage{caption}

 \usepackage{multirow}
\usepackage{xcolor}

\newcommand{\change}[1]{\textcolor{black}{#1}}

\DeclareUnicodeCharacter{3000}{}
\DeclareUnicodeCharacter{FF0C}{}

\usepackage{hyperref}

\floatname{algorithm}{Algorithm}

\begin{document}

\title{The Mobile Production Vehicle Routing Problem: Using 3D Printing in Last Mile Distribution}

\author[label1]{Yu Wang}
\address[label1]{Department of Applied Mathematics, Xi'an Jiaotong-Liverpool University, Suzhou, Jiangsu 215123, China}

\ead{yu.wang1803@student.xjtlu.edu.cn}

\author[label2]{Stefan Ropke}
\address[label2]{DTU Management, Technical University of Denmark, Akademivej Building 358, 2800 Kgs. Lyngby, Denmark}
\ead{ropke@dtu.dk}

\author[label1]{Min Wen\corref{cor1}}

\ead{min.wen@xjtlu.edu.cn}

\cortext[cor1]{Corresponding authors}

\author[label3]{Simon Bergh}
\address[label3]{3D Printhuset A/S - Skudehavnsvej 17A, 2150 Nordhavn}
\ead{skb@3dprinthuset.dk}

\begin{frontmatter}

\begin{abstract}
We study a new variant of the vehicle routing problem, called the Mobile Production Vehicle Routing Problem (MoP-VRP). In this problem, vehicles are equipped with 3D printers, and production takes place on the way to the customer. The objective is to minimize the weighted cost incurred by travel and delay of service. We formulate a Mixed Integer Programming (MIP) model and develop an Adaptive Large Neighbourhood Search (ALNS) heuristic for this problem. To show the advantage of mobile production, we compare the problem with the Central Production Vehicle Routing Problem (CP-VRP), where production takes place in a central depot. We also propose an efficient ALNS for the CP-VRP. We generate benchmark instances based on Vehicle Routing Problem with Time Windows (VRPTW) benchmark instances, and realistic instances based on real-life data provided by the Danish Company 3D Printhuset. Overall, the proposed ALNS for both problems are efficient, and we solve instances up to 200 customers within a short computational time. We test different scenarios with varying numbers of machines in each vehicle, as well as different production time. The results show that these are the key factors that influence travel and delay costs. The key advantage of mobile production is flexibility: it can shorten the time span from the start of production to the delivery of products, and at the same time lower delivery costs. Moreover, long-term cost estimations show that this technology has low operation costs and thus is feasible in real life practice.
\end{abstract}


\end{frontmatter}

\section{Introduction}

In the modern supply chain system, balancing the interests of manufacturers, distributors, and customers has always been a difficult problem. Manufactures and Distributors want to minimize their operation costs. Customers, meanwhile want to receive a good quality of service. Among the various costly factors behind the operations, inventory management is a long-time problem that has complicated the production and circulation of commodities. Annual reports at Amazon, for example, have listed the inventory problem as a risk factor within the operational process for two consecutive years (\citeauthor{amazon17}, \citeyear{amazon17}; \citeauthor{amazon18}, \citeyear{amazon18}). In 2017, Amazon reported that a 1 percent inventory increase would cost an additional 180 million USD (\citeauthor{amazon17}, \citeyear{amazon17}). With the rapid development of industrial manufacturing, an effective solution called make-to-order (MTO) was proposed to reduce such inventory costs. 

In MTO mode, when the central warehouse receives an order, production begins there directly. The number of goods to be produced is the same as the number of orders received (i.e., no excess good is made). When production is completed, trucks deliver the assigned goods to the customers, as shown in Figure \ref{ex_cp}. To minimize the total cost, the practitioners want to find the optimal production scheduling and transportation scheme. This need has raised an optimization problem, which we have named \textit{the Central Production-Vehicle Routing Problem (CP-VRP)}. 

Although the on-demand central production in MTO mode can eliminate inventory costs, the potential drawback is that it is not flexible enough because waiting for the completion of production may delay delivery to the customers. This disadvantage does not conform to customers' wish for a high-quality service, which may ultimately lead to a decline in both service quality and brand reputation. 

To improve and innovate the logistics model, Amazon applied for a patent (US Patent 9684919) in 2015 (\citeauthor{amazontech}, \citeyear{amazontech}; \citeauthor{Geek}, \citeyear{Geek}) that combines 3D printing technology and MTO mode. With this patent, a car-mounted 3D printer produces the products on the way to the customer, as shown in Figure \ref{ex_mop}. By applying this technique, production and transportation are synchronized and delivery times can be largely reduced.

\begin{figure}[htbp]
\centering
\begin{minipage}[t]{0.48\linewidth}
\includegraphics[width=0.8\textwidth]{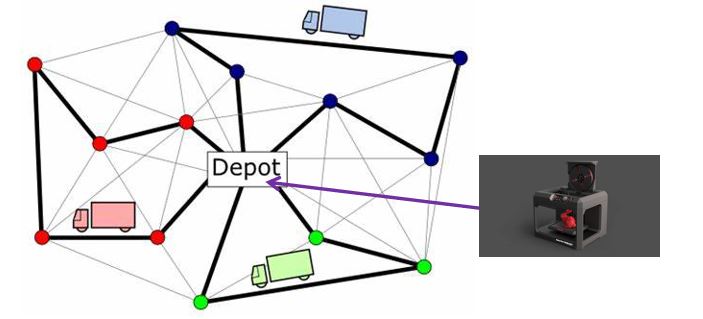}
\caption{The Central Production}
\label{ex_cp}
\end{minipage}
\begin{minipage}[t]{0.32\linewidth}
\includegraphics[width=0.8\textwidth]{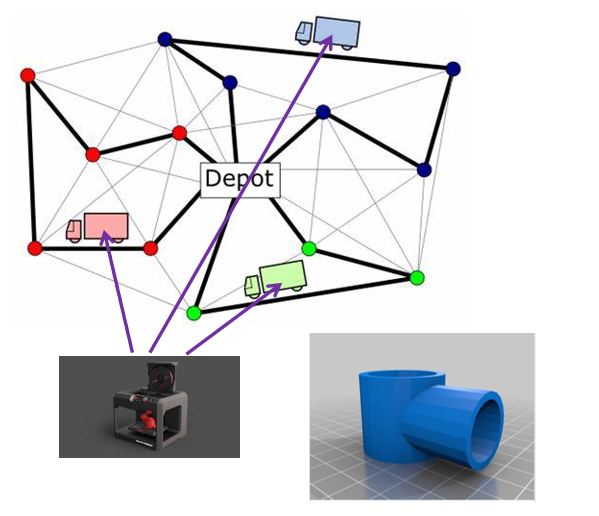}
\caption{The Mobile Production}
\label{ex_mop}
\end{minipage}
\end{figure}

This zero-inventory, synchronous production, and transportation technology can meet people’s demand for customization and prompt delivery, all while reducing the cost of logistics. Thus, applying this technology can be a good strategy to balance the interests of manufacturers, distributors, and customers in the modern supply chain system. With the further upgrading of industrial manufacturing and the continuous maturation of 3D printing technology, the era of intelligent manufacturing and personalized customization is coming. Thus, this new technology will have great potential for large-scale use in the future.

The integration of the car-mounted 3D-printing factory and the vehicle routing problem is denoted as the \textit{Mobile Production Vehicle Routing Problem (MoP-VRP)}. Unlike the CP-VRP, in which delivery starts only after production is completed at the depot, production and distribution in the MoP-VRP take place synchronously. This setting makes the MoP-VRP more complex than the CP-VRP, as production and delivery are more closely connected. This close connection also makes the existing methods proposed for the CP-VRP inapplicable to the MoP-VRP, thus we need to find a new suitable algorithm to solve the MoP-VRP. 

The proposed MoP-VRP considers a limited homogeneous vehicle fleet. A certain number of machines are installed in each vehicle. Production starts at the time the vehicle departs from the depot. Each customer has a strict starting time window and a soft ending time window. The vehicle should return to the depot before a certain time. \change{The objective is to minimize the weighted sum of travel and delay cost, where the travel cost is based on the total travel length and the delay cost is based on the delayed time. }

In this paper, we introduce the MoP-VRP and analyse its properties. We present mixed integer programming models for the MoP-VRP and the CP-VRP. We develop an Adaptive Large Neighbourhood Search (ALNS) algorithm for both problems. By analysing the computational results for the two problems, we illustrate the advantages of mobile production.

Most of the MoP-VRP's applications will be in the future, when 3D printing technology has advanced to a state where popular consumer goods can be 3D printed on demand. However, an application that, in principle, could be rolled out at the time of writing is that of producing spare parts needed for repairs. Most of the machinery and equipment that we rely on in everyday life contains plastic parts that can break down. Examples of such equipment include everything from coffee machines, printers, and photocopiers to plumbing parts or even cars, buses, and planes. When a plastic part breaks down, the user may be unable to use the machine or equipment and have to wait for a repair. That repair may have to wait until spare parts are shipped from a distant warehouse, and for older equipment the spare parts may not even be available. If the spare part can be printed by a 3D printer mounted in a vehicle, the customer may get the spare part on the day she makes her request or even within hours if the MoP-VRP is solved in a dynamic setting (the present paper only considers the static MoP-VRP; the dynamic version is left for future work). 

\change{3D printers have already been used on board ships to print spare parts for repairs while at sea  (\citeauthor{brian}, \citeyear{brian}). One may be concerned about the quality of printing as the production environment is likely to be unstable both at sea and on the road. To solve this potential risk, \citeauthor{PHILLIPS2020100969} (\citeyear{PHILLIPS2020100969}) used passive stabilization to help on board 3D printers provide equally good quality products as the traditional land-based laboratory. Based on this, we find it likely that 3D printing on board a moving van is feasible using appropriate passive stabilization. It should be noted, that the technology that we envision to be used in a spare parts delivery scenario is fused filament fabrication (FFF), which from the beginning is more robust to movement compared to the stereolithography (SLA) technology studied by \citeauthor{PHILLIPS2020100969} (\citeyear{PHILLIPS2020100969})}.

The computational results in section \ref{sec:ResultsRealistic} show that the cost for each delivery can be kept at a reasonable level if a decent amount of daily customers can be attracted to the service. The biggest challenge, in rolling out the envisioned MoP-VRP spare-parts service may be in building a vast library of 3D models that correspond to possible spare parts and in obtaining the rights to print spare parts for many different manufacturers.

The contribution of this work is five-fold: 1) We introduce the MoP-VRP, a new, non-trivial variant of the vehicle routing problem; 2) we study properties of the MoP-VRP and CP-VRP that allows us to accelerate heuristics for the two problems; 3) we design and implement Adaptive Large Neighborhood heuristics for both problems; 4) we propose testing instances for the two problems including a set of benchmark instances and a set of realistic instances in cooperation with 3D Printhuset; 5) we perform an extensive computational experiment to show advantages of the two forms of production and routing.

The rest of the paper is organized as follows: a literature review on related works is presented in section 2. The problem setting is described in section 3; based on this, two integer programming models, necessary assumptions, definitions, and properties for the defined problem are presented. We introduce the ALNS algorithm in section 4 and present computational results along with findings in section 5. Finally, concluding remarks are summarized in section 6.

\section{Literature Review}

This section presents a review of 1) the Central Production Vehicle Routing Problem and 2) mobile production technology.

\subsection{The Central Production Vehicle Routing Problem}

In order to effectively use resources, the optimization of production scheduling and transportation schemes in the classical MTO model have been widely studied. Different works may deal with unique problem settings, and the name given to the solved problem can vary from work to work. In the review by \citeauthor{moons2017integrating} (\citeyear{moons2017integrating}), this name is generalised as \textit{the Production Scheduling and Vehicle Routing Problem (PS-VRP)}. The most commonly seen names in the literature are the integrated/joint production-distribution scheduling problem (\citeauthor{belo2015adaptive}, \citeyear{belo2015adaptive}; \citeauthor{guo2017harmony}, \citeyear{guo2017harmony}; \citeauthor{mohammadi2020integrated}, \citeyear{mohammadi2020integrated}; \citeauthor{yaugmur2020memetic}, \citeyear{yaugmur2020memetic}) and the integrated production scheduling and vehicle routing problem (\citeauthor{chen2005integrated}, \citeyear{chen2005integrated}; \citeauthor{zou2018coordinated}, \citeyear{zou2018coordinated}). In this work, we name it as \textit{the Central Production Vehicle Routing Problem (CP-VRP)}.

With respect to optimization strategy, many enterprises apply a decomposition method in practice to consider production scheduling and transportation schemes independently and optimize them separately (\citeauthor{scholz2011approach}, \citeyear{scholz2011approach}; \citeauthor{moons2017integrating}, \citeyear{moons2017integrating}). While the decomposition method reduces the difficulty of solving the problem, it is hard to obtain the global best solution due to the split of problem, and the obtained solution can be quite different from the optimal one (\citeauthor{moons2017integrating},\citeyear{moons2017integrating}). To use resources more efficiently and improve service quality, more and more companies are beginning to study the integrated optimization strategy, which is to combine central production and transportation decisions into a single optimization problem (\citeauthor{moons2017integrating}, \citeyear{moons2017integrating}; \citeauthor{chang2004machine}, \citeyear{chang2004machine}). 

In recent years, in-depth research has been conducted on the CP-VRP at various levels, and has theoretically proved that the zero-inventory MTO model can effectively help companies eliminate risks in inventory expenses and cash flow. In the literature, a large number of variants of the CP-VRP have been studied, and the corresponding mathematical models or theoretical derivations have been proposed (\citeauthor{amorim2013lot}, \citeyear{amorim2013lot}; \citeauthor{jamili2016bi}, \citeyear{jamili2016bi}). The variants of the problem can be classified according to the production mode and further classified as single machine production (\citeauthor{tamannaei2019mathematical}, \citeyear{tamannaei2019mathematical}; \citeauthor{jamili2016bi}, \citeyear{jamili2016bi}; \citeauthor{low2014coordination}, \citeyear{low2014coordination}; \citeauthor{li2016integrated}, \citeyear{li2016integrated}; \citeauthor{zou2018coordinated}, \citeyear{zou2018coordinated}), parallel machine configuration (\citeauthor{amorim2013lot}, \citeyear{amorim2013lot}; \citeauthor{dayarian2019branch}, \citeyear{dayarian2019branch}; \citeauthor{kesen2019integrated}, \citeyear{kesen2019integrated}; \citeauthor{belo2015adaptive}, \citeyear{belo2015adaptive}; \citeauthor{ullrich2013integrated}, \citeyear{ullrich2013integrated}; \citeauthor{guo2017harmony}, \citeyear{guo2017harmony}), production in the job shop (\citeauthor{mohammadi2020integrated}, \citeyear{mohammadi2020integrated}), and production in the flow shop (\citeauthor{yaugmur2020memetic}, \citeyear{yaugmur2020memetic}; \citeauthor{wang2019variable}, \citeyear{wang2019variable}). Some of these studies have added the classic time window constraints for vehicle routing problems (\citeauthor{amorim2013lot}, \citeyear{amorim2013lot}; \citeauthor{guo2017harmony}, \citeyear{guo2017harmony}; \citeauthor{mohammadi2020integrated}, \citeyear{mohammadi2020integrated}). The CP-VRP and MoP-VRP solved in this work are most relevant to the variant studied by \citeauthor{kesen2019integrated} (\citeyear{kesen2019integrated}), where parallel machine configuration and a limited homogeneous vehicle fleet are considered. We extend this model a little by adding the duration of routes.

The mathematical model and valid inequality for the CP-VRP have been widely studied in the literature. \citeauthor{chang2004machine} (\citeyear{chang2004machine}) used mathematical derivation to prove that the CP-VRP problem, with only one production machine, is an NP-hard problem. \citeauthor{geismar2008integrated} (\citeyear{geismar2008integrated}) proved that the CP-VRP problem is an NP problem and gave the lower bound of the solved problem by mathematical proof. \citeauthor{armstrong2008zero} (\citeyear{armstrong2008zero}) proved that the use of the basic properties of the tackled problem can accelerate the branch-and-bound search. \citeauthor{chen2005integrated} (\citeyear{chen2005integrated}) defined CP-VRP with different scenarios and gave mathematical proof of the complexity of some of the defined scenarios. 

The exact algorithm for CP-VRP has also been investigated in the literature. \citeauthor{amorim2013lot} (\citeyear{amorim2013lot}) verified the importance of introducing lot-sizing by solving the mixed integer programming model. \citeauthor{tamannaei2019mathematical} (\citeyear{tamannaei2019mathematical}) dealt with a problem with production due date and used the branch-and-bound algorithm to solve the problem. In this work, the 
``earliest due date first production" scheduling scheme is proposed and is proven to effectively accelerate the search for the lower bound. \citeauthor{dayarian2019branch} (\citeyear{dayarian2019branch}) designed an effective branch-and-cut-and-price algorithm and introduced new branching rules to efficiently obtain the lower bound of each branch.

Most of the literature uses heuristic or meta-heuristic algorithms to solve CP-VRP related problems. The most used are the neighbourhood search algorithm (\citeauthor{belo2015adaptive}, \citeyear{belo2015adaptive}; \citeauthor{jamili2016bi}, \citeyear{jamili2016bi}; \citeauthor{wang2019variable}, \citeyear{wang2019variable}), genetic algorithm (\citeauthor{tamannaei2019mathematical}, \citeyear{tamannaei2019mathematical}; \citeauthor{kesen2019integrated}, \citeyear{kesen2019integrated}; \citeauthor{yaugmur2020memetic}, \citeyear{yaugmur2020memetic}; \citeauthor{ullrich2013integrated}, \citeyear{ullrich2013integrated}; \citeauthor{low2014coordination}, \citeyear{low2014coordination}; \citeauthor{li2016integrated}, \citeyear{li2016integrated}; \citeauthor{guo2017harmony}, \citeyear{guo2017harmony};  \citeauthor{zou2018coordinated}, \citeyear{zou2018coordinated}), and decomposition method (\citeauthor{dayarian2019branch}, \citeyear{dayarian2019branch}; \citeauthor{ullrich2013integrated}, \citeyear{ullrich2013integrated}; \citeauthor{zou2018coordinated}, \citeyear{zou2018coordinated}). Among them, \citeauthor{jamili2016bi} (\citeyear{jamili2016bi}) and \citeauthor{li2016integrated} (\citeyear{li2016integrated}) used the Pareto-Optimal Solution to provide a trade-off between transportation costs and service quality. In our study, we apply the Adaptive Large Neighbourhood Search (ALNS) algorithm to solve both the MoP-VRP and CP-VRP. Moreover, we propose some strategies to accelerate the ALNS for the CP-VRP.

Another existing optimization problem groups together the production and routing is the \textit{Production Routing Problem (PRP)} (\citeauthor{adulyasak2015production}, \citeyear{adulyasak2015production}; \citeauthor{qiu2018variable}, \citeyear{qiu2018variable}), which is an integrated optimization problem that jointly optimizes production, inventory, and routing decisions. The PRP usually deals with a multiple-period planning horizon, and in each period the number of products to be made, the number of products to be stored in the depot (i.e, the inventory), and the delivery plan must be decided. Different from the PRP, Both the CP-VRP and the MoP-VRP in this work will not consider overproduction (i.e., no inventory cost), and we only deal with a one-day planning horizon. 

\subsection{The Mobile Production Technology}

To the best of our knowledge, only two works introduce the concept of mobile production technology in the literature; different works define a different kind of mobile production.

\citeauthor{malladi2020dynamic} (\citeyear{malladi2020dynamic}) introduced a mobile production mode in which production is conducted by a movable production unit called a “module”. The number of modules determines the capacity of production, and the target is to optimize the allocation of modules, production, and inventory planning to serve uncertain demand within a certain planning horizon. 

In \citeauthor{pasha2020optimization} (\citeyear{pasha2020optimization}), the mobile production is named the ``factory in box". First, trucks pick up raw materials from the supplier and semi-produced products from the factory. Then, production takes place at the factory or customer's location. \change{Although the vehicle carries the ``factory" and moves during the delivery process, production still takes place at a fixed point, not en route.} The mobility of production in this work is reflected in the flexible location of the factory. 

The problem studied by \citeauthor{malladi2020dynamic} (\citeyear{malladi2020dynamic}) and \citeauthor{pasha2020optimization} (\citeyear{pasha2020optimization}) differs from the proposed MoP-VRP. Production in these two works takes place at a stationary point, whereas production in the MoP-VRP takes place en route.

\section{Problem Description and Mathematical Model}

This section will describe the setting of the MoP-VRP and the CP-VRP. Then, the mathematical model for both problems will be given. After that, some properties of the MoP-VRP and CP-VRP, as well as corresponding proof are provided.

\subsection{Problem Description}

This section describes the setting of the MoP-VRP and the CP-VRP. We consider both problems within the context of urban delivery. 

Both the MoP-VRP and CP-VRP are defined on a graph as follows: $G = (V, E)$ with $V = \{0\} \cup C$ as the set of nodes and $E$ as the set of edges defined between connected nodes. Node 0 is the depot and \change{$C = \{1, \ldots, n\}$} is the set of customers. There exists a homogeneous set of vehicles \change{$K = \{1,\ldots,\kappa\}$}, each with capacity $Q$. We consider a fixed number of vehicles in this work and multiple trips are not allowed. The vehicle will start from the depot and visit a sequence of customers, then go back to the depot no later than a certain time. \change{The vehicle can depart at time 0 and should go back to the depot no later than $D$. } Each customer $i \in C$ has demand $d_i$ and production time $p_i$. Each customer also has a pre-specified time window [$a_i$, $b_i$]. We consider a strict starting time window and a soft ending time window for both problems. The service time for each customer $i$ is denoted as $e_i$. The distance from $i$ to $j$ is denoted as $c_{ij}$, and the traveling time from $i$ to $j$ is denoted as $t_{ij}$. We denote $M = \{1,\ldots,m\}$ as the set of machines on each vehicle that produce the goods for the customers. \change{For the CP-VRP, we define $M' = \{1,\ldots,m'\}$ as the set of machines in the depot. In this work, we assume that each vehicle has the same number of machines, so $m' = m*\kappa$. }\change{In practice, there is a setup/clean up time associated with each job, but this can be included in the production time and is therefore not modelled explicitly.}

For the MoP-VRP, a number of machines are installed in each vehicle. Production and delivery start simultaneously at the beginning of the planning horizon. The service start time at each location is the largest value among the arrival time, the production completion time, and the starting time window for that location. \change{With the current technology, there is a need for human intervention every time a new printing job is started. This means that the driver will have to stop the vehicle and spend a minute or two starting the new job, which could be problematic while driving on the highway. This is an issue that is not modelled currently, but is left for future work.} For the CP-VRP, the machines are installed in the depot, and each vehicle leaves the depot only when all the products assigned to the vehicle are ready. 

The aim for both problems is to determine an optimal integrated production and delivery schedule such that the weighted sum of the total travel distance and the total delay is minimized. 


\begin{figure}[h]
\centering
\includegraphics[width=0.8\textwidth]{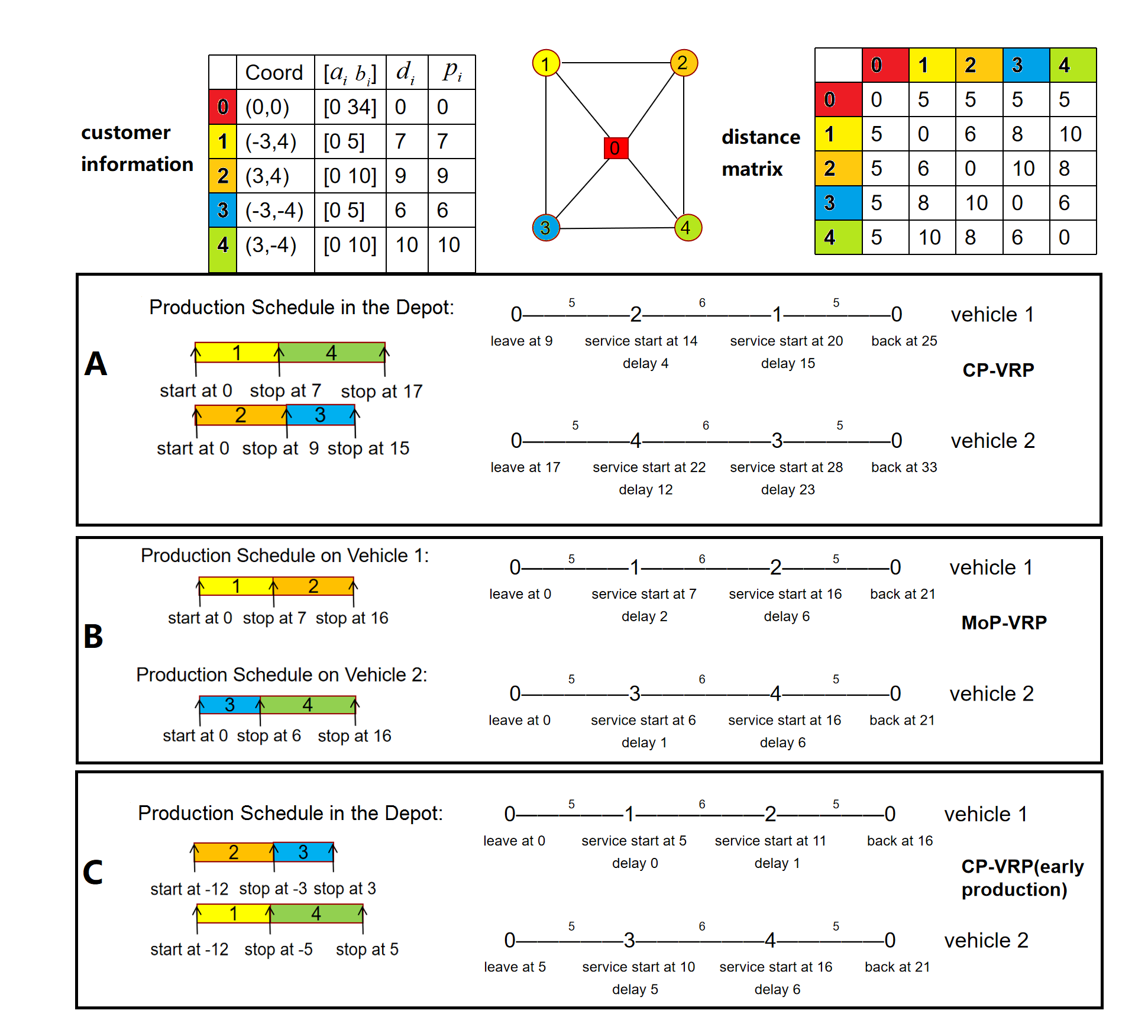}
\caption{An Example to compare the MoP-VRP and the CP-VRP (Coord: the coordinate; [$a_i,b_i$]: the time window; $d_i$: the demand; $p_i$: the production time)\label{toy}}
\end{figure}

Figure \ref{toy} shows an example of an MoP-VRP and a CP-VRP solution. In the example, we have two machines in total and two available vehicles. In the CP-VRP, the two machines are placed in the depot; in the MoP-VRP we have one machine on each vehicle. Each vehicle has capacity 16 and the latest time to return the depot is 34. The demand, production time, time windows, and coordinates for each customer are shown in Figure \ref{toy}. The service time for each customer is 0. The graph and the distance matrix are also shown at the top of the figure. \change{The weight for both the travel cost and delay cost is set as 1.}

Clusters A and B in Figure \ref{toy} show the optimal solution for the CP-VRP and the MoP-VRP, where production starts at 0 in both problems. In the MoP-VRP, the vehicles leave the depot at 0 and production takes place on the way to the customer. However, in the CP-VRP solution, production takes place at the depot, vehicles have to wait until the production is finished before they can leave the depot. It is found that in the CP-VRP vehicle 1 will leave the depot as soon as the assigned products (1, 2) are complete, while vehicle 2 needs to wait until product 4 is made. The optimal solution for both problems has the same travel distance of 32 (5 + 6 + 5 + 5 + 6 + 5), the MoP-VRP just has a delay of 15 (2 + 6 + 1 + 6), and the CP-MoP has a delay of 54 (4 + 15 + 12 + 23). Therefore, as we set the weight for both terms to 1, the objective value for the MoP-VRP is 47 (32 + 15), which is much lower than that for the CP-VRP (86 = 32 + 54).

One way to reduce the delay to the customer could be to allow the machines at the depot to start production at an earlier time\change{, say $H$ time units}. Therefore, we introduce early production to the CP-VRP to help reduce delay costs. \change{Note that the vehicle should depart the depot no earlier than 0, even if all the assigned products are finished before 0.}

The cluster C in Figure \ref{toy} shows the optimal solution for the CP-VRP, with early production considered, in which production begins 12 time units before 0 \change{(i.e., $H$ = 12)}. It is found that the total delay can be largely reduced to 12 (0 + 1 + 5 + 6) and is lower than that in the optimal solution for the MoP-VRP. However, the drawback of early production is that: the entire operation now spans a longer time period. We return to the consequences of this drawback in Section 5.

\subsection{Mathematical Model}

In this section, we present mixed integer programming (MIP) models for both the MoP-VRP and the CP-VRP.

In addition to the parameters introduced in Section 3.1, we define $\delta^+(i)$ and $\delta^-(i)$ as the set of nodes that can be reached from $i\in V$ and the set of nodes that can reach $i\in V$, respectively. We define a dummy start node $o$ and dummy end node $d$ at the depot 0. 

The binary variable $x^k_{ij}$ equals 1 if vehicle $k \in K$ travels from $i$ to $j$, and 0 otherwise. Variable $y_i$ denotes the delay time for customer $i \in C$. The non-negative variable $s^k_i$ is the time vehicle $k \in K$ which starts delivery at \change{node $i \in V$}. In the MoP-VRP model, the non-negative variable $v^k_{il}$ is the production starting time for \change{node $i \in V$} on machine $l \in M$ in vehicle $k$ (if $i$ is served by $k$) and the binary variable $w^k_{ijl}$ equals 1 if customer $j$'s order is produced immediately after customer $i$'s on machine $l \in M$ in vehicle $k \in K$.

The mathematical model for the MoP-VRP is set as:

\begin{align}
    	\text{min} \quad &\omega_1 \sum_{k \in K}\sum_{(i,j) \in E} c_{ij} x^k_{ij} + \omega_2 \sum_{i \in C}y_{i}　 \label{eq:obj}\\
	    \text{s.t.}\quad & \sum_{k \in K}\sum_{j\in\delta^{+}(i)}x^k_{ij} = 1 &\forall i \in C \label{eq:visit}\\
	    & \sum_{j \in \delta^{+}(o)}x^k_{oj} = \sum_{j \in \delta^{-}(d)}x^k_{jd} = 1 &\forall k \in K \label{eq:depot}\\ 
	    &\sum_{j \in \delta^{+}(i) }x^k_{ij} = \sum_{j \in \delta^{-}(i)}x^k_{ji} &\forall i \in C, k \in K \label{eq:flowconserv}\\
	    &\sum_{i \in C}\sum_{j \in \delta^{+}(i) }d_i x^k_{ij} \leq Q &\forall k \in K \label{Capcity_mop}
	    \\
	    &\sum_{l \in M}\sum_{j \in \delta^{+}(i)}w^k_{ijl} = \sum_{j \in \delta^{+}(i)}x^k_{ij} &\forall i \in C, k \in K \label{eq:linkxw}\\
	    &\sum_{j \in \delta^{+}(0)}w^k_{0jl} = 1 &\forall l \in M, k \in K \label{eq:startendforprod}\\
	    &\sum_{j \in \delta^{+}(i) }w^k_{ijl} = \sum_{j \in \delta^{-}(i)}w^k_{jil} &\forall i \in C , l \in M, k \in K \label{eq:flowconservforprod}\\
	    &v^{k}_{jl} \geq v^{k}_{il} + p_i - \tilde{M}(1-w^k_{ijl}) &\forall (i,j) \in E, l \in M, k\in K \label{eq:productstart}\\
	    &s^{k}_i \geq v^{k}_{il} + p_i &\forall i \in C, l \in M, k \in K \label{eq:prodserv}\\
	    &s^k_j \geq s^k_i + t_{ij} + e_i - \tilde{M}(1-x^k_{ij}) &\forall (i,j) \in E, k \in K \label{eq:travel}\\
	    &s^k_i \geq a_i &\forall i \in C, k \in K \label{eq:startTW_MoPVRP}\\
	    &s^{k}_{d} \leq D &\forall k \in K \label{eq:maxdur}\\
	    &y_{i} \geq s^{k}_{i} - b_{i} &\forall i \in C, k \in K \label{eq:delay}\\
	    &x^k_{ij}, w^k_{ijl} \in \{0,1\} &\forall (i,j) \in E, l\in M, k \in K\label{eq:xw}\\
	    &y_{i} \geq 0 &\forall i \in C\label{eq:y}\\
	    &v^{k}_{il},s^k_i \geq 0 &\forall i \in V, l \in M, k \in K\label{eq:vs}
\end{align}

\change{The objective \eqref{eq:obj} is to minimize the weighted sum of travel and delay cost, where travel cost is based on the total travel length and the delay cost is based on the delayed time. In this work, we set the weight for both terms as 1.} Constraints \eqref{eq:visit} ensure that each customer is visited exactly once by one vehicle. Constraints \eqref{eq:depot} guarantee that each vehicle starts and ends at its starting and ending depots. Constraints \eqref{eq:flowconserv} are the flow conservation constraints. Constraints \eqref{Capcity_mop} ensure that the capacity is not exceeded for each vehicle. Constraints \eqref{eq:linkxw} make sure that if a customer $i$ is visited by vehicle $k$, then this order should be produced on one of the machines on $k$. Constraints \eqref{eq:startendforprod}-\eqref{eq:flowconservforprod} are the production flow conservation constraints for each machine. Constraints \eqref{eq:productstart} determine the start production time for each order. If order $j$ is produced immediately after order $i$ by machine $l$ in vehicle $k$ (i.e., $x^k_{ij}=1$), then the start production time of order $j$ should be greater than or equal to the start production time of order $i$ plus the production time of $i$, which is $v^{k}_{jl} \geq v^{k}_{il} + p_i $. Constraints \eqref{eq:prodserv} and \eqref{eq:travel} ensure that the delivery can start only if the vehicle has travelled to the customer and the production of that order is finished. \change{We suggest that $\tilde{M}$ should be a value no smaller than $D+\max\{\max_{i \in C}\{p_i \},\max_{(i,j) \in E}\{t_{ij}+e_i\}\} $.} Constraints \eqref{eq:maxdur} set the maximum travel duration for each vehicle. Constraints \eqref{eq:delay} determine the delay time for each customer. Constraints \eqref{eq:xw}-\eqref{eq:vs} are the non-negativity constraints.

In the CP-VRP, $w^k_{ijl}$ and $v^k_{il}$ changes to $w_{ijl}$ and $v_{il}$ because production does not take place on the vehicle. $v_{il}$ can have a negative value because early production is allowed in this problem.

\begin{align}
    	\text{min} \quad &\omega_1 \sum_{k \in K}\sum_{(i,j) \in E} c_{ij} x^k_{ij} + \omega_2 \sum_{i \in C}y_{i}　 \label{eq:obj_cprp}\\
	   \text{s.t.}\quad  & \eqref{eq:visit}-\eqref{Capcity_mop} \notag
	    \\
	    &\sum_{l \in M'}\sum_{j \in C}w_{ijl} = 1 &\forall i \in C \label{eq:w_cprp}\\
	    &\sum_{j \in \delta^{+}(0)}w_{0jl} = 1 &\forall l \in M \label{eq:startforprod_cprp}\\
	    &\sum_{j \in \delta^{+}(i) }w_{ijl} = \sum_{j \in \delta^{-}(i)}w_{jil} &\forall i \in C , l \in M' \label{eq:flowconservforprod_cprp}\\
	    &v_{jl} \geq v_{il} + p_i - \tilde{M}(1-w_{ijl}) &\forall (i,j) \in E, l\in M' \label{eq:productstart_cprp}\\
	    &s^{k}_{o} \geq v_{il} + p_i - \tilde{M}(1-\sum_{j \in \delta^{+}(i)}x^k_{ij}) &\forall i \in C, l \in M', k \in K \label{eq:prodserv_cprp}\\
	    &v_{il} \geq -H &\forall i \in C, l \in M' \label{AheadProd}\\
	    & \eqref{eq:travel}-\eqref{eq:vs} \notag
\end{align}

In the CP-VRP, constraints \eqref{eq:startforprod_cprp}-\eqref{eq:productstart_cprp} should be interpreted like \eqref{eq:startendforprod}-\eqref{eq:productstart} in the MoP-VRP. Constraints \eqref{eq:w_cprp} ensure that each customer's demand should be produced by some machine. Constraints \eqref{eq:prodserv_cprp} guarantee that each vehicle can only start when all the assigned products are finished. Constraints \eqref{AheadProd} set the early production.

\subsection{Properties of the MoP-VRP and CP-VRP}

It is easy to see that MoP-VRP is an NP-Hard problem: we consider an instance for the MoP-VRP, where the production time for each customer is 0. It is sufficient to show that to solve this special instance of an MoP-VRP is equivalent to solving a VRP.

Next we turn to a simple property of the MoP-VRP that allows us to discard a part of the feasible solution space when designing heuristics for the problem. To explain this property, it is convenient
to start with a small example for an MoP-VRP with two machines per vehicle. Consider the vehicle route {[}1 2 3 4 5 6{]}, where each element is the ID of a customer. A possible machine schedule for this route
is Machine 1: {[}1 2 5{]} and Machine 2: {[}4 3 6{]}, where the elements of the vectors indicate which customers the machine is producing for. We propose that the schedule for machine 1 is \emph{in line with } the route since the machine completes jobs in the same order as deliveries are being made to the customers. Machine 2's schedule, meanwhile, is not \emph{in line with }the route because the jobs for customer 4 and 3 are produced in the opposite order of their appearance in the route. 

It is obvious that changing the order of customers 3 and 4 in the schedule for machine 2 can never make the solution worse: when product 4 is produced before 3, neither of the two customers can be served before
the production of both orders is finished. When product 3 is produced before 4, customer 3 can be served once product 3 has been finished while customer 4 cannot be served before both products are done; this implies that with production schedule [3 4 6], the vehicle will serve customer 3 and 4 (or the customers thereafter) at the same time or earlier than production schedule [4 3 6]. This is formalized into the following proposition:

\begin{proposition}
There exists an optimal solution for the MoP-VRP, such that the schedule of each machine is in line with the route of the truck carrying the machine.
\end{proposition}

\begin{proof}
Suppose that \change{$s^{*}$}  is the unique optimal solution, where customer 2 is served after customer 1, but his product is produced before customer 1 on the same machine, and another solution \change{$\hat{s}$} where the sequence of production and delivery are the same for two customers. We define \change{$f(s)$} as the objective value for solution \change{$s$}. As the route is fixed, the service for customer 1 in \change{$\hat{s}$} will start no later than that in \change{$s^{*}$}, and the service starting time for customer 2 is the same in both solutions. Therefore, \change{$f(\hat{s}) \leq f(s^{*})$ }; thus, \change{$s^{*}$} is not the unique optimal solution. This is contradictory to the assumption.
\end{proof}
\medskip{}
The following proposition states that the CP-VRP can be decomposed into a production planning problem and a routing problem when the CP-VRP only involves a single vehicle

\begin{proposition}
When there is one vehicle and $m$ machines in the CP-VRP, the optimal solution can be found by first minimizing the production completion time and then solving a TSP. 
\end{proposition}

\begin{proof}
As multiple trips are not allowed in this work and the vehicle cannot leave before the completion of all the assigned items, and since the production order does not influence the route of vehicle, the problem can be split in two as indicated in this proposition.
\end{proof}

We next investigate the structure of the schedule for a single machine in the CP-VRP. The goal of this analysis is to allow us to rule out certain machine schedules from the set of schedules to consider. This will allow us to design faster heuristics in section \ref{Sec:ALNS}. We note that the machine schedules influence the overall solution through the times when products are available. All products to be delivered
by a single route must be available before the vehicle can depart; poor machine scheduling will therefore lead to unnecessary delay costs.

The key insight is that if the schedule for a particular machine contains two or more products that are to be delivered on the same route, then these products can be produced in direct succession. For example, consider a solution where the production schedule for a machine in the CP-VRP is $[\begin{array}{cccccc}
1^{(2)} & 2^{(3)} & 3^{(4)} & 4^{(3)} & 5^{(4)} & 6^{(2)}\end{array}]$, where each element in the vector has the form $i^{(r)}$, in which $i$
is the ID of the customer who will receive the product and $r$ is the route that this customer is assigned to. Starting from the back of the schedule, we can gather the two products that are to be
delivered by route 2. This gives the schedule $[\begin{array}{cccccc}
2^{(3)} & 3^{(4)} & 4^{(3)} & 5^{(4)} & 1^{(2)} & 6^{(2)}\end{array}]$, which is at least as good as the starting schedule because the products for route 2 are finished at the same time as in the original schedule while the products for route 3 and 4 are finished earlier. We now continue backwards through the schedule to the next ``scattered'' products which are the products to be delivered on route 4. We gather
these to obtain the schedule $[\begin{array}{cccccc}
2^{(3)} & 4^{(3)} & 3^{(4)} & 5^{(4)} & 1^{(2)} & 6^{(2)}\end{array}]$. In this schedule, the products for route 2 and 4 are finished at the same time as in the previous schedule while the products to be
delivered by route 3 are finished earlier. We now have a schedule where products that are to be delivered by the same route are in direct succession. 

Lemma \ref{lemma:gather-jobs} below formalizes the process of gathering two products destined
for the same route. In this Lemma, we let $r(i)$ denote the route that serves customer $i$. The proof of the Lemma follows directly from the arguments above.

\begin{lemma}\label{lemma:gather-jobs}
Consider a production schedule 
$[\begin{array}{cccc} i_{1} & i_{2} & \ldots & i_{\theta}\end{array}]$, where $\theta$ is the length of the schedule and $i_{j}$ indicates the customer ID that the machine is producing for. If there exists $\alpha,\beta,\gamma$ such that $1\leq\alpha<\beta<\gamma\leq \theta$ and such that $r(i_{\alpha})=r(i_{\gamma})$ and $r(i_{\alpha})\neq r(i_{\beta})$, then job $i_{\alpha}$ can be
moved to position $\gamma-1$ and the jobs that previously occupied position $\alpha+1$ to $\gamma-1$ will now occupy position $\alpha$ to $\gamma-2$. This change will not delay departure for any routes
and may reduce the departure time for some routes.
\end{lemma}

From Lemma \ref{lemma:gather-jobs} we derive the following proposition:

\begin{proposition}
There exists an optimal solution for the CP-VRP, where the production schedule for each machine finishes all jobs delivered by one route before switching to producing jobs that are to be delivered by a different
route. 
\end{proposition}

\begin{proof} Consider any optimal solution to the CP-VRP instance. If the optimal solution does not satisfy the conditions of proposition 3, then we process each machine schedule that does not satisfy the conditions by repeatedly applying Lemma \ref{lemma:gather-jobs}. We go through the machine
schedule from the back and gather jobs destined for the same route by repeatedly applying Lemma \ref{lemma:gather-jobs}. After having applied Lemma \ref{lemma:gather-jobs} \change{$\eta$} times, we are sure that the \change{$\eta$} last jobs are grouped according to their delivery route and we therefore at most need to apply Lemma \ref{lemma:gather-jobs} as many times as there are jobs in the schedule. Applying Lemma \ref{lemma:gather-jobs} does not make the cost of the solution any worse so we will therefore
eventually end up with an optimal solution in which jobs are grouped according to their delivery route for each schedule. This shows that an optimal solution with the desired properties exists. 
\end{proof}

\if false
We introduce an iterated method that transforms the production schedule to show that the production schedule can have alternative form. In the first iteration, we fix the last customer in the production schedule (suppose this customer belongs to route $r$) and move all the products that are assigned to route $r$ to the end of the production schedule. In the next iteration, we fix the last customer in the production schedule before the cluster of route $r$ (suppose this customer belongs to route $r'$) and then we produce all the products belong to route $r'$ before the cluster of route $r$. We repeat this reschedule process until there is no other customer before the current cluster. We use an example to show the process:

Suppose the current production schedule is [2 3 4 3 4 2], where 2 3 4 are the route ID the customer is assigned to. 

In the first iteration, as the last customer belongs to route 2, we produce all the products belong to route 2 at the end and then the schedule becomes [3 4 3 4 2 2].

In the second iteration, the last customer before the cluster of route 2 belongs to route 4, so we transform the production schedule to [3 3 4 4 2 2].

Finally, as all the customers belong to route 3 has gathered and there is no other customer before the cluster of route 3. We complete the transformation of production schedule.

\begin{lemma}
The iterated method to rearrange the production schedule in an optimal solution will make the form of production schedule become $[\ldots C^*_i \quad C^*_j \ldots]$, where $C^*_i$ represents the cluster of customers that are assigned to the same route $i$. This form will not cause the increase of departure time for any routes compared to the original production schedule.
\end{lemma}

\begin{proof}
In the first iteration, as we fix the last customer in the production schedule, the production completion time for the route the last customer assigned to keeps the same. Then, to move other customers belong to the same route to the tail of production schedule may advance the production completion time for the other routes. Thus, the final form of production schedule will not lengthen the departure time for any routes compared to the original schedule.
\end{proof}

\begin{proposition}

There exists an optimal solution for the CP-VRP where the production schedule is the same as the form the iterated method obtained.

\end{proposition}

\begin{proof}

From lemma 1, we know that this form of production schedule will not cause the increase of departure time for any routes, thus the delay cost by the final form will never be larger than that by the original production schedule. Therefore, if the original schedule leads to an optimal solution, the final form is also an alternative optimal solution.

\end{proof}
\fi

The following proposition compares the objective value of the MoP-VRP and the CP-VRP when there is exactly one vehicle available.
\begin{proposition}
In the case when there is one vehicle and $m$ machines available, the objective value of the optimal MoP-VRP solution is always better than or equal to the objective value of the optimal solution to the corresponding CP-VRP instance, assuming that the early production parameter $H$ is set to 0.
\end{proposition}

\begin{proof}
If the route and production schedule from an optimal CP-VRP solution is copied to an MoP-VRP solution, the production completion time for the customers is the same for the two problems. Since the vehicle in the CP-VRP will start from the depot after the completion of all the products, the service starting time for each customer in the MoP-VRP cannot be later that in the CP-VRP. Thus, with the same solution, the MoP-VRP will never obtain a higher total cost than the CP-VRP because the travel cost is the same and the delay cost is at least as good as that in the CP-VRP.

\end{proof}

When the number of vehicles is greater than 1, the optimal solution to the CP-VRP can be better than the optimal solution to the MoP-VRP, even when the early production parameter $H$ is set to 0. To illustrate, consider 4 customers and a depot distributed on a line as shown on Figure \ref{fig:customersOnALine}. Numbers on top of the line indicates coordinates; letters below indicate customers. The depot is located at coordinate 0. Assume that travel times and distances are equal to the Euclidean distances, and that service times are 0. Further, assume that the production time for customers A and B is 20, while it is 1 for customers C and D, and that the ending time windows are 31 for customers A and D and 30 for customers B and C. 

\vspace*{-4mm}
\begin{figure}[htbp]
\begin{center}
\includegraphics[width=0.5\textwidth]{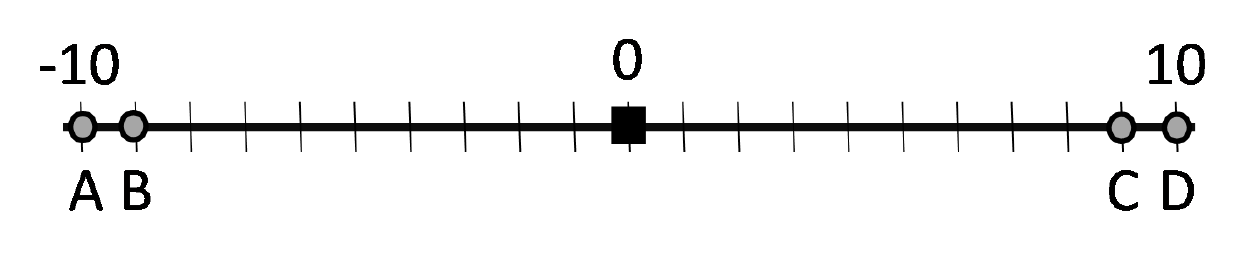}
\end{center}
\vspace*{-10mm} 
\caption{Example where the CP-VRP solution is better than the MoP-VRP solution.\label{fig:customersOnALine}}
\end{figure}
\vspace*{-4mm}
Assuming that we have two vehicles available, each equipped with 1 machine in the MoP-VRP, then an optimal solution is to place A and B on route 1 and C and D on route 2. The total traveled distance is 40, and there will be a delay of 9 time units for customer A (total objective: 49). For the CP-VRP we have two machines at the depot. These machines finish producing all products at time 21. Reusing the two routes from the MoP-VRP leads to zero delays in the CP-VRP case, so the total objective is 40.

\section{Adaptive Large Neighbourhood Search}\label{Sec:ALNS}

We developed an Adaptive Large Neighbourhood Search (ALNS) heuristic to solve the large-size problems efficiently. The ALNS (\citeauthor{ropke2006adaptive},\citeyear{ropke2006adaptive}) searches in large neighborhoods defined implicitly by a destroy method and a repair method. A destroy method disrupts part of the current solution while a repair method rebuilds the destroyed solution. In the ALNS, a set of destroy and repair methods are developed for selection. Each is assigned a weight that is adaptively adjusted based on their performance in previous iterations. In each iteration, a destroy and repair method is selected and applied to the current solution. The resultant solution is accepted based on a user-defined acceptance criterion, and the heuristic stops when the stopping criterion is met. 

The pseudo-code of our ALNS is presented in Alg.\ref{Alg:ALNS}. In line 2, an initial solution is constructed by a greedy insertion heuristic. In lines 6-8, we choose whether to use the noise or not based on the previous performance; after that we choose the destroy and repair operator. In the solution searching phase, we apply the threshold acceptance method proposed by \citeauthor{dueck1990threshold} (\citeyear{dueck1990threshold}) which has been shown as a good criterion for the ALNS (\citeauthor{santini2018comparison},\citeyear{santini2018comparison}). We always accept the better solution but also accept a deteriorated solution if it is within a given threshold $T$ of the best solution found so far. The parameter $T$ is initialized in the beginning of the search and decreased linearly to 0 at the end of the search. We will tune the initial threshold $T_{initial}$ in our final test to ensure the best performance of the algorithm. The ALNS stops when a certain number of predefined iterations, $N_{max}$, is reached.  

To enable the comparison between MoP-VRP and CP-VRP, we also apply ALNS to the CP-VRP. The pseudo-code of Alg.\ref{Alg:ALNS} also applies to the CP-VRP. 

\begin{algorithm}[h]
\begin{algorithmic}[1]
\State{\textbf{Function:} ALNS($instance$)}
\State{$s = \text{greedyinsertion}(instance)$ ;}
\State{$s^* = s$;}
\Repeat
\State{$s' = s$;}
\State{decide whether to use the noise or not;}
\State{select a destroy operator and update $s'$;}
\State{select a repair operator and update $s'$;}
\change{
\If{$\displaystyle\frac{f(s')-f(s^*)}{f(s^*)} < T $}
\State{$s = s'$;}
\If{$f(s') < f(s^*)$}
\State{$s^*=s'$;}
\EndIf
\EndIf
\If{$N_{max} \% 100 = 0$}
\State{Update weights for the destroy, repair operators and noise;}
\EndIf
\State{$T = T - (T_{initial}/N_{max})$;}}
\Until{$N_{max}$ is met}
\end{algorithmic}
\caption{Adaptive Large Neigbourhood Search} \label{Alg:ALNS}
\end{algorithm}

\subsection{The ALNS for the MoP-VRP}

This section will describe the strategy to construct the solution for the MoP-VRP and the removal and insertion operators used in the ALNS.

\subsubsection{Construction of Solution}

We used the parallel insertion heuristic to construct the initial solution. We start with a set of empty routes. In each iteration, we evaluate the cost of inserting each unassigned customer to each possible route position as well as each possible production position in the assigned route. We then identify the minimum insertion cost for each unassigned customer, select the one with the smallest minimum insertion cost, and insert it to the best position. The procedure is repeated until all the customers are included in the routes.  

Different from the traditional VRP, in the MoP-VRP, not only does the routing position need to be decided for each customer but also the machine production schedule for the corresponding customer demand. For each inserted routing position, one needs to examine all the possible production schedules for each machine. However, as proven in Proposition 1, there exists an optimal solution such that the production sequence on the same machine is consistent with the route sequence. Following Proposition 1, we only consider the production whose sequence is \textit{in line with} that of route in our implementation. This helps to reduce computational time significantly. The pseudo-code for the insertion strategy is shown in Alg. \ref{Alg:mop_best}. 

\begin{algorithm}[h]
\begin{algorithmic}[1]
\For{each non-inserted Customer $i$}
\For{each Route \change{$r$}}
\For{each position \change{$\rho$} $\in$ Route \change{$r$}}
\State{calculate travel increment on position \change{$\rho$}}
\For{each Machine \change{$l$} on Route \change{$r$}}
\State{Find the last customer whose delivery position in route \change{$r$} is before \change{$\rho$}}
\State{Compute the delay increment in route \change{$r$} for the customers after position \change{$\rho$}}
\State{\change{$(i^*, \rho^*, l^*, r^*) = (i, \rho, l, r)$} if the total cost increment is the least}
\EndFor
\EndFor
\EndFor
\EndFor
\State{\change{Return $(i^*, \rho^*, l^*, r^*)$}}
\end{algorithmic}
\caption{The insertion strategy for the MoP-VRP} \label{Alg:mop_best}
\end{algorithm}

We apply the worst-case analysis to calculate the computational complexity for Alg. \ref{Alg:mop_best}. It is found that the computational complexity is \change{$O(n^2m\frac{n}{\kappa})=O(\frac{mn^3}{\kappa})$}, where $n$ is the number of customers, $m$ is the number of machines in each vehicle, and \change{$\kappa$} is the number of vehicles in the instance. The $n^2$ comes from the truth that the loops in lines 1-3 result in $n^2$ repetitions, the $m$ comes from line 5, which represents the number of machines in each vehicle, and \change{$\frac{n}{\kappa}$} is the maximum number of repetitions resultant from line 7. Under the assumption that \change{$\kappa = O(n)$}, the computational complexity can be simplified to \change{$O(mn^2)$}.

\subsubsection{Destroy and Repair Operators}
In this work, we propose six destroy operators and four repair operators. The MoP-VRP and CP-VRP share the same destroy and repair operators introduced in this section.

Common to all destroy methods is that we first randomly pick a value \change{$\Phi \in [\lfloor \lambda_1 n \rfloor,\lfloor \lambda_2 n \rfloor]$} as the number of customers to remove. For the interval \change{$[\lfloor \lambda_1 n \rfloor,\lfloor \lambda_2 n \rfloor]$, $n$} is the total number of customers and \change{$\lambda_1$ and $\lambda_2$} are the minimum and maximum ratios, respectively. We will tune the lower bound and upper bound (i.e., \change{$\lambda_1$ and $\lambda_2$}) in our final test.

The random and worst removals use the same strategy as that introduced by \citeauthor{ropke2006adaptive} (\citeyear{ropke2006adaptive}). The \textbf{random removal} picks \change{$\Phi$} customers at random to remove. \change{The \textbf{worst removal} removes \change{$\Phi$} customers with the highest cost savings $\xi_{i}$. Inspired by the worst removal, \textbf{Worst-Delay Removal} removes $\Phi$ customers with the highest delay cost savings, and \textbf{Worst-Dist Removal} removes $\Phi$ customers with the highest travel distance savings. Inspired by the existing Shaw removal (\citeauthor{shaw1997new},\citeyear{shaw1997new}; \citeauthor{shaw1998using},\citeyear{shaw1998using}), we also propose two similar destroy operators, \textbf{Geo Removal} and \textbf{Demand Removal}, which first remove a random seed customer and then $\Phi-1$ customers that are close to the seed customer in terms of distance and demand, respectively. Details of the destroy operators are presented in the Appendix. }

For the repair operators, the proposed ALNS applies the \textbf{regret-k method}. The idea is to choose the insertion that we will regret most if we do not perform it. \change{Let $r_{ih}$ indicate the route where customer $i$ has the $h$th lowest insertion cost and $\Delta_{i,r_{ih}}$ be the inserting cost of customer $i$ on route $r_{ih}$. The regret value $g^{*}_{i} $ equals to $\sum^k_{h=1}\{\Delta_{i,r_{ih}} - \Delta_{i,r_{i1}}\}$. The regret-k method will choose to insert customer $i$ that maximizes: $ \text{arg max}\{g^{*}_{i} \}$ }and will put it into the position with the lowest cost increment. Ties are broken by selecting the request with the best insertion cost. In this work, we implement the regret-1/-2/-3/-4 method, regret-1 method is the greedy insertion introduced in section 5.1 due to the tie-breaking rule.

We also introduce noise to give randomness to the objective value every time we calculate the insertion cost of a request. We calculate the noise by applying the same strategy used by \citeauthor{ropke2006adaptive} (\citeyear{ropke2006adaptive}). 

Each removal and insertion heuristics is given a weight that can be adaptive adjusted to measure how well each operator performs on each instance. We also give the automatically adjusted weight to the noise so that the algorithm itself can decide whether to apply the noise or not in each iteration. A higher weight indicates a more effective operator, and an operator with a higher weight is more likely to be chosen in each iteration. The mechanism and parameters to adjust the weight are exactly the same as \citeauthor{ropke2006adaptive} (\citeyear{ropke2006adaptive}).

\subsection{The ALNS for the CP-VRP}

The ALNS for the CP-VRP shares the same framework as the ALNS for the MoP-VRP, and the destroy/repair operators for the MoP-VRP also apply to the CP-VRP. We introduce some new strategies to accelerate the ALNS for the CP-VRP.

\subsubsection{Construction of Solution}
We also apply the parallel insertion heuristic to construct the initial solution for the CP-VRP. We start with a set of empty routes. In each iteration, we evaluate the cost of inserting each unassigned customer to all the possible route positions and fix the route position with the smallest cost increment. We then identify the smallest total cost increment by trying all the possible positions in the production sequences at the depot. We select the customer with the smallest minimum insertion cost and add it to the route and machine schedule in the solution. The procedure is repeated until all the customers are included in the solution.

\subsubsection{The Potential Production Sequence for the CP-VRP}

When we identify the best production position for a customer whose routing position is already determined, we should examine all the possible positions in the production sequence for each machine to ensure that the total cost increment is the global minimum. However, testing all the possible positions can make the computation quite time-consuming. From Proposition 3, we know that grouping production for customers who are assigned to the same vehicle to a shared machine can lead to an optimal solution. We can use this property to eliminate the cost computation of a large number of unnecessary production positions, thereby reducing the computational time significantly. 

Then, we provide an example to show how we narrow the search for potential production positions. Suppose that a new customer has already been selected to be added to route $r$, and we now need to consider when the corresponding production should take place for each machine. Assume the machine currently produces products for five customers. Let's use a vector to indicate the route index of each customer in the corresponding production sequence. For example, [4 4 2 2 1] means that the first two customers for whom the machine produces are delivered to by vehicle 4, the next two customers by vehicle 2, and the last one by vehicle 1. Then, we may meet two cases:

\textbf{Case 1}: If the new customer is served by a vehicle that is not in the vector, e.g., $r = 3$, it would make sense to just try the positions marked with *: $[* \quad  4 \quad 4 \quad * \quad 2 \quad 2 \quad * \quad 1 \quad * ]$. There is no need to try positions like: $A = [4 \quad * \quad 4 \quad 2 \quad 2 \quad 1]$, as they can never lead to a better solution than this position: $B = [* \quad 4 \quad 4 \quad 2 \quad 2 \quad 1]$.

The reason is that for both $A$ and $B$, the ready time for routes 1, 2, 4 is the same. $A$, however, will lead to a higher ready time for route 3.

\textbf{Case 2}: If the new customer is served by one of the vehicles that is already in the vector, e.g., $r = 4$, we could try these potential positions (marked with both * and $\circ$): $[* \quad  4 \quad 4 \quad 2 \quad 2 \quad \circ \quad 1 \quad \circ]$. When we try positions marked with $\circ$, such as: $[4 \quad 4 \quad 2 \quad 2 \quad \circ \quad 1]$, we should also compute the cost of $V = [2 \quad 2 \quad 4 \quad 4 \quad 4 \quad 1]$ because this schedule is at least as good as $W = [4 \quad 4 \quad 2 \quad 2 \quad 4 \quad 1]$.

This is because for both $V$ and $W$, the ready time for route 4 is the same but $V$ can lead to an earlier departure time for route 2, which is cost saving.

\subsubsection{The Piece-wise Linear Function for the CP-VRP}

An insertion or removal of one product in the production schedule may affect the departure time for several vehicles. If the departure time is changed, the delay cost for the route may change. Further, it is time-consuming to go through all the positions in the affected route to calculate the change in the delay cost. 

When the route is fixed, we find that the total delay of the route can be expressed as a non-decreasing piece-wise linear function \change{($Delay(\psi)$)} of the route departure time, where input \change{$\psi$} is the departure time and \change{$Delay(\psi)$} is the corresponding delay cost. The function has a typical shape, as given in Figure \ref{piecewise}. With the help of \change{$Delay(\psi)$}, we can store the information of how the delay cost changes with the departure time. Then, each time we obtain a new departure time, we can easily obtain the new delay cost and thus the computational time can be largely reduced. 

To introduce how the \change{$Delay(\psi)$} helps save computational time, we present an example of a route in Figure \ref{example_piecewise}, and its corresponding piece-wise linear function in Figure \ref{piecewise}. In Figure \ref{example_piecewise}, the two boxes illustrate the start and end of the route, and the circles illustrate the customers visited. The numbers in square brackets above the customers are the time windows. The numbers under the edges are travel times, and the numbers in parentheses are service times. This route would not be feasible in an ordinary VRPTW setting because it would be impossible to visit customer 2 after having visited customer 1 with the given time windows. However, in both the CP-VRP and MoP-VRP, such a solution is feasible, because delay is allowed. In Figure \ref{piecewise}, we see that the slope is different in each interval; this represents the different number of customers whose delay costs change with the departure time. The numbers shown in the x-axis, 15 25 65, are the thresholds that mark the change of increment rate. Here, 65 is the maximum departure time to exceed that will cause an infeasible solution. Delay can be calculated as follows: if the departure time becomes 8, $Delay(8) = 15$, which is equal to $Delay(0)$; if the departure time becomes 20, $Delay(20) = 20$.

When analysing the computational complexity of the piece-wise linear function, it takes \change{$O(n_r)$} to construct the piece-wise linear function. There is no need to reconstruct the function if the route does not change, and it takes \change{$O(log(n_r))$} to compute the delay if the departure time does change \change{($n_r$ is the number of customers in the route $r$)}. Without the piece-wise linear cost function, it would take \change{$O(n_r)$} to compute the delay once the departure time has changed. Therefore, the proposed function can help accelerate the algorithm.

\begin{figure}[htbp]
\centering
\begin{minipage}[h]{0.48\linewidth}
\includegraphics[width=0.8\textwidth]{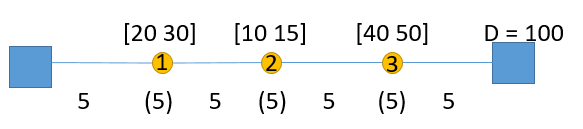}
\caption{Example to explain the piece-wise linear function \label{example_piecewise}}
\end{minipage}
\begin{minipage}[h]{0.48\linewidth}
\includegraphics[width=0.8\textwidth]{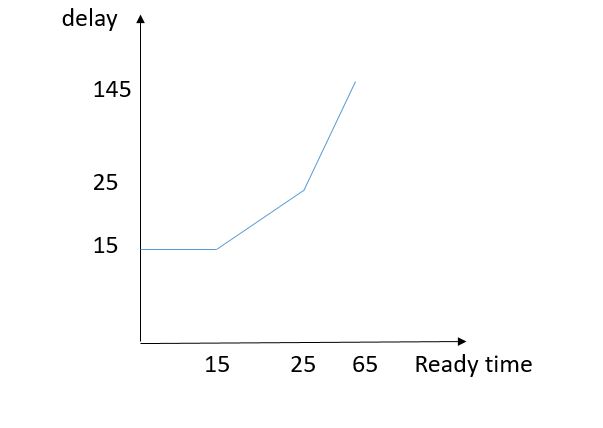}
\caption{Piece-wise Linear Function \label{piecewise}}
\end{minipage}
\end{figure}
\subsubsection{Comparison of different insertion strategy}

The most accurate way to calculate the cost increment when we insert a customer is shown in Alg. \ref{Alg:int}, by which we need to try all the positions in each route. For each delivery position, we need to try all the potential production sequence to ensure that the final customer insertion causes the smallest increment cost. We call this strategy the integrated strategy. 
\begin{algorithm}[h]
\begin{algorithmic}[1]
\For{each non-inserted customer $i$}
\For{each Route \change{$r$}}
\For{each position \change{$\rho$ in route $r$}}
\State{calculate travel increment at position \change{$\rho$}}
\State{Update the Piece-wise Linear Function for route \change{$r$} temporarily}
\For{each machine \change{$l$}}
\For{each position \change{$\nu$ in the schedule of $l$}}
\State{Compute the delay increment for the whole solution}
\State{\change{$(i^*, r^*, \rho^*, l^*, \nu^*) = (i, r, \rho, l, \nu)$} if the total cost increment is the least}
\EndFor
\EndFor
\EndFor
\EndFor
\EndFor
\State{\change{Return $(i^*, r^*, \rho^*, l^*, \nu^*)$}}
\end{algorithmic}
\caption{The integrated strategy for the CP-VRP} \label{Alg:int}
\end{algorithm}

In line 5 in Alg. \ref{Alg:int}, we need to update the piece-wise linear function for route $r$ because the route has changed and the pre-stored cost function cannot be used anymore. We update the function temporarily, and it will turn back to the original form when we finish the computation for position \change{$\rho$} in route \change{$r$}. We also apply the worst-case analysis to estimate the computational complexity of Alg. \ref{Alg:int}. It is found that the computational complexity for Alg. \ref{Alg:int} is \change{$O(n^2(\frac{n}{\kappa}+m'\kappa^2log(\frac{n}{\kappa} + 1))) = O(\frac{n^3}{\kappa}+n^2m'\kappa^2log(\frac{n}{\kappa} + 1)$}, where $n$ is the number of customers, \change{$m'$} is the total number of machines in the depot, and \change{$\kappa$} is the number of vehicles in the instance. In the computational complexity, $n^2$, \change{$\frac{n}{\kappa}$}, and \change{$m'$} come from lines 1-3, line 5, and line 6, respectively. The \change{$\frac{n}{\kappa}$} comes from the analysis that to have the same number of customers in each route will cause the highest computational complexity. The \change{$\kappa^2log(\frac{n}{\kappa} + 1)$} is composed of two parts: \change{$\kappa$ and $\kappa log(\frac{n}{\kappa} + 1)$, where $\kappa$} comes from line 7 because there will be at most \change{$\kappa$} positions to consider due to the analysis from section 4.2.2, and \change{$\kappa log(\frac{n}{\kappa} + 1)$ }comes from line 8, where the piece-wise linear cost function is applied. Under the assumption that \change{$\kappa = O(n)$}, the computational complexity can be simplified to \change{$O(n^2 + m'n^4)$}. We can omit the $O(n^2)$ part, so the final computational complexity is \change{$O(m'n^4)$}.

To save the computational time, we apply a decomposition strategy for the CP-VRP. The pseudo-code for the proposed strategy is shown in Alg.\ref{Alg:FindBestInsertion}. 

\begin{algorithm}
\begin{algorithmic}[1]
\For{each non-inserted Customer $i$}
\For{each Route \change{$r$} }
\State{Find insertion of Customer $i$ in route \change{$r$} (assume that route ready time is not changed by insertion)}
\State{Save insertion if total cost increment is the best}
\EndFor
\State{Update the Piece-wise Linear Function for the best route}
\For{each Machine \change{$l$}}
\For{each position \change{$\nu$ in the schedule of Machine $l$}}
\State{Compute the delay increment for the whole solution}
\State{Save the insertion if delay increment is the best}
\EndFor
\EndFor
\State{\change{$(i^*, r^*, \rho^*, l^*, \nu^*) = (i, r, \rho, l, \nu)$} if total cost increment is the best}
\EndFor
\State{\change{Return $(i^*, r^*, \rho^*, l^*, \nu^*)$}}
\end{algorithmic}
\caption{The decomposition strategy} \label{Alg:FindBestInsertion}
\end{algorithm}

It can be seen from Alg. \ref{Alg:FindBestInsertion} that we fix the delivery sequence at first without considering the change of ready time. We then find the best production schedule that minimizes the cost increment for all vehicles. Computational complexity is largely reduced because we separate the loop for the routes and the machines. By applying the worst-case analysis, the computational complexity for Alg. \ref{Alg:FindBestInsertion} is \change{$O(\frac{n^3}{\kappa} + nm'\kappa^2log(\frac{n}{\kappa} + 1))$, where $\frac{n^3}{\kappa}$ comes from lines 1-5 and $nm'\kappa^2log(\frac{n}{\kappa} + 1))$ comes from lines 7-12. Under the assumption that $\kappa = O(n)$, the computational complexity can be simplified to $O(n^2 + mn^3)$. We can omit the $O(n^2)$ part, so the final computational complexity is $O(m'n^3)$. By applying this strategy, we can reduce the computational complexity from $O(m'n^4)$ to $O(m'n^3)$.}

We need to redefine the approach to find $k$ best routes in regret-k insertion because it is different from that in the integrated method. As shown in Alg. \ref{Alg:newregretk}, for each customer, we choose the $k$ routes with the lowest travel increment first. Then, we update the cost increment for the picked $k$ routes by searching for the best production schedule. After that, we sort the $k$ routes in ascending order of the updated cost increment.

\begin{algorithm}[h]
\begin{algorithmic}[1]
\For{each Route $r$}
\State{Find insertion of $i$ in route $r$ (assume that route ready time is not changed by insertion)}
\State{Save insertion if travel increment is best}
\EndFor
\State{Find the $k$ routes with lowest travel increment}
\For{each route $r \in$ the $k$ Best Routes}
\For{each Machine $l$ }
\For{each position in the schedule of Machine $l$}
\State{Compute the delay increment for the whole solution}
\State{Save the insertion if delay increment is the best}
\EndFor
\EndFor
\State{Save the cost increment for route $r'$}
\EndFor
\State{Sort the chosen $k$ routes in ascending order of cost increment}
\end{algorithmic}
\caption{The best $k$ routes for each customer in the decomposition strategy} \label{Alg:newregretk}
\end{algorithm}

\section{Computational Results}

To evaluate the performance of the proposed mathematical model and ALNS, a series of experiments have been conducted on two different instance sets. \change{The codes have been written in a workstation by using C++ with Windows 10, Intel Core i9-7940X, 3.10GHz, 32 GB RAM. CPLEX 12.8 is used to solve the MIP models.}

\subsection{Instances}
We generate two sets of instances. One set is based on the well-known VRPTW benchmark Solomon Instance (\citeauthor{Solomon1987Algorithms}, \citeyear{Solomon1987Algorithms}) and the Gehring and Homberger's extended VRPTW Instance (\citeauthor{gehring1999parallel}, \citeyear{gehring1999parallel}), where we use the same vehicle capacity, customer locations, duration, time windows, service times, and demands as in the benchmark instances. The other set is constructed to mimic a real-life scenario where mobile production is employed.

We create different scenarios by changing the production time and number of machines per vehicle. The number of machines per vehicle is either 1, 2, or 4. The production time $p_i$ for customer $i$ is determined by the formula $p_i = \mu d_i$, where $\mu$ is the production time for each demand unit (we name it production time coefficient) and is set as either 1, 3, or 5, and $d_i$ is the demand for customer $i$.

To set the maximum number of vehicles in each instance, we run a greedy sequential insertion algorithm that inserts all customers in a solution. The number of routes this algorithm terminates with is the maximum number of vehicles of the instance.

For the realistic instances, we consider a delivery scenario in Copenhagen, Denmark as the setting. We generate data using the spare-parts-for-repairs case, as described in the introduction.

\begin{figure}[htbp]
\centering
\begin{minipage}[h]{0.38\linewidth}
\includegraphics[width=0.8\textwidth]{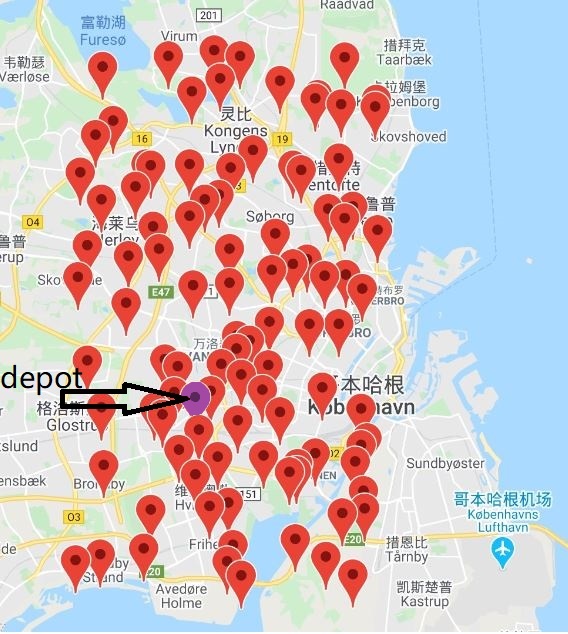}
\caption{Real Life Instance\label{5}}
\end{minipage}
\begin{minipage}[h]{0.38\linewidth}
\includegraphics[width=0.8\textwidth]{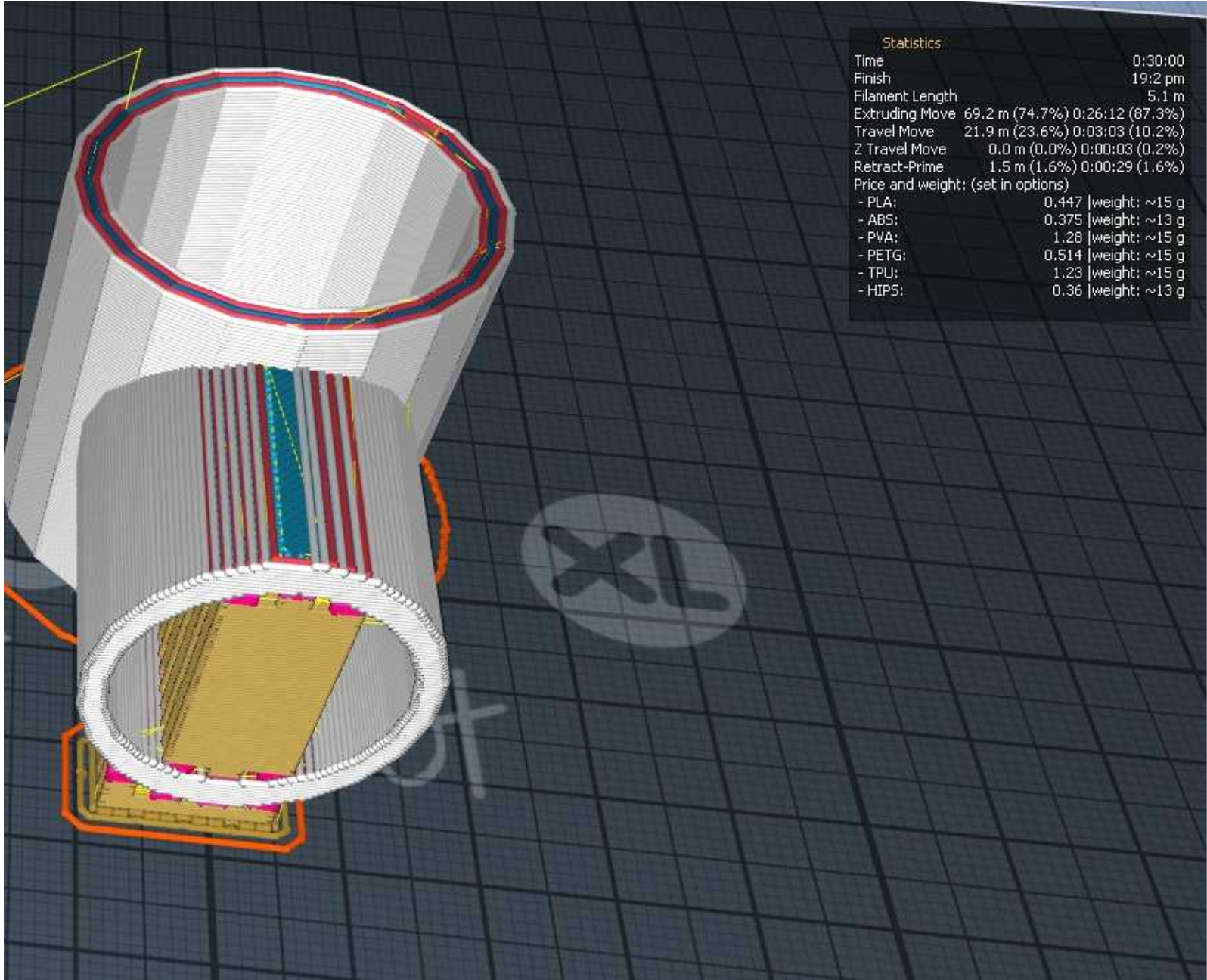}
\caption{Production Example\label{6}}
\end{minipage}
\end{figure}

As shown in Figure \ref{5}, we randomly generate 100 coordinates in a 20 by 30 km rectangle in the greater Copenhagen area. One of them is chosen as the depot and the rest as customers. The distances between the positions are generated by using Google Maps API, and the travel times are calculated using the distances divided by an average speed of 50 km/h. We assume that the drivers, working hours are from 11 am to 9 pm (i.e. 600 minutes) per day.  

We also generate 25 and 50 customer instances, as the company may not receive many orders at the beginning stage of this new business mode due to immature technology or the small market. For the 25 and 50 customer scenarios, we randomly pick 25 and 50 customers, respectively, from the 99 customers. The service time for each customer is a random number within 1-5 minutes. 

We create six scenarios in total, each differing in terms of the type of production time and time windows. We have three types of production time: small production time (abbreviated S), which ranges from 20 to 30 minutes; medium production time (30-40 mins, abbreviated M); and high production time (30-60 mins, abbreviated H). The three types of production times are estimated by supplying different 3D models of various spare parts to the Craftbot slicing program (Craftware, see Figure \ref{6}). The program allows us to compute the estimated printing time for any given 3D model. 

We create two types of time windows: wide time windows (abbreviated W) and tight time windows (abbreviated T). The length of the time window is between 30 and 60 minutes for the wide time window, and is within 30 minutes for the tight time window. 

Based on the above setting, we named the six scenarios as follows: short production time with wide time window (S\_W), medium production time with wide time window (M\_W), high production time with wide time window (H\_W), short production time with tight time window (S\_T), medium production time with tight time window (M\_T) and high production time with tight time window (H\_T). For each scenario, we generate five random instances with random production times and time windows. 

The instances are available at \url{https://zenodo.org/record/4892330#.YLczY4Xitdg}.

During the implementation, we define the early production time as \change{$H = \frac{\epsilon P}{m\kappa}$, where $P = \sum^{n}_{i = 1} p_i$ is the total production time and $\epsilon$ is the early production coefficient. We set $\epsilon = 0.75$ } and will explain why this value is chosen in Section 5.4. 

\subsection{Parameter Tuning}

In this section, we report results from parameter tuning on the initial threshold and the number of customers to remove, which from our experience, are very important to the algorithm. We randomly pick five instances for each size (25, 50, 100, 200) (i.e., 20 instances in total) to conduct the parameter tuning test, and we let the ALNS run 10 times for each instance for each parameter setting. The values shown from Table \ref{InitTemp} to Table \ref{noremove200_cp} represent the gap between the \change{average objective value $\bar{z}$ over 10 runs for each parameter and the best objective value $z^*$ found in the 10 runs among all the parameters, calculated as $\frac{\bar{z} - z^*}{z^*}\times100\%$.}

During the tuning for the initial threshold, the range to remove customers is set as (25\%-50\%). Table \ref{InitTemp} and Table \ref{InitTemp_cp} show the influence of the initial threshold $T_{initial}$. When $T_{initial}$ is too low or too high, the heuristic has a higher chance of being trapped in one suboptimal area of the search space. The results show
that $T_{initial} = 10\%$ and $T_{initial} = 17.5\%$ are the best choices for MoP-VRP and CP-VRP respectively.

\begin{table}[h]
\tiny
\resizebox{\textwidth}{5.5mm}{
\begin{tabular}{ccccccccccccc}
\hline
Initial Threshold/\% & 1      & 2.5    & 5      & 7.5    & \textbf{10}     & 12.5   & 15     & 17.5   & 20     & 30     & 40     & 50     \\
\hline
Avg Gap/\%             & 2.18 & 1.96 & 1.16 & 1.33 & \textbf{0.83} & 0.98 & 0.99 & 1.04 & 1.39 & 1.35 & 1.70 & 2.03\\
\hline
\end{tabular}
}
\caption{Tuning of Initial Threshold (MoP-VRP)} \label{InitTemp}
\end{table}

\begin{table}[h]
\tiny
\resizebox{\textwidth}{5.5mm}{
\begin{tabular}{ccccccccccccc}
\hline
Initial Threshold/\% & 1      & 2.5    & 5      & 7.5    & 10     & 12.5   & 15     & 17.5   & 20     & 30     & 40     & 50     \\
\hline
Avg Gap/\%             & 5.09 & 4.32 & 3.87 & 3.08 & 3.16 & 2.99 & 3.05 & \textbf{2.87} & 2.99 & 3.17 & 3.22 & 3.34\\
\hline
\end{tabular}
}
\caption{Tuning of Initial Threshold (CP-VRP)} \label{InitTemp_cp}
\end{table}

When we tuned the number of customers to remove, we first test the Solomon 25, 50, and 100 instances, then pick the best five ranges to continue the tuning for 200 customer instance. We tune in this way to save computational time because it is very time-consuming to test 200 customer instances. For the MoP-VRP, as shown in Table \ref{noremove100}, the best five ranges for the instances with up to 100 customers are (20\%-50\%), (10\%-40\%), (10\%-45\%), (25\%-45\%), and (25\%-35\%). So they are picked to continue the tuning. Table \ref{noremove200} shows the average results for these five ranges on 25, 50, 100, and 200 customer instances. We find that with the increase of $\lambda_1$, the average gap widens; it is also not good to set $\lambda_2$ as a too large (50\%) or too small (35\%) value. Among these ranges, (10\%-40\%) obtains the smallest gap and therefore is selected as the range of removing customers for the MoP-VRP.

\begin{table}[h]
\centering
\small
\begin{tabular}{llllll}
\hline
\diagbox{$\lambda_1$/\%}{$\lambda_2$/\%}
  & 50     & 45 & 40 & 35     & 30     \\
   \hline
5  & 0.82\%          & 0.71\%      & 1.00\%      & 1.07\% & 1.10\% \\
10 & 0.87\%          & 0.55\%      & 0.52\%      & 0.81\% & 1.13\% \\
15 & 0.76\%          & 0.64\%      & 0.66\%      & 0.74\% & 0.94\% \\
20 & \textbf{0.40\%} & 0.77\%      & 0.81\%      & 0.87\% & 1.01\% \\
25 & 0.80\%          & 0.56\%      & 0.65\%      & 0.63\% & 0.97\%
\\
\hline
\end{tabular}
\caption{Tuning of range of customers to remove (Solomon 25, 50, 100) (MoP-VRP)} \label{noremove100}
\end{table}

\begin{table}[h]
\centering
\small
\begin{tabular}{cccccc}
\hline
Range & (20\%, 50\%) & \textbf{(10\%, 40\%)} & (10\%, 45\%)&  (25\%, 45\%) &  (25\%, 35\%)    \\
\hline
Avg Gap/\%             & 1.36 & \textbf{0.65} & 0.7 & 2.07 & 2.01\\
\hline
\end{tabular}
\caption{Tuning of range of customers to remove (25, 50, 100, 200) (MoP-VRP)} \label{noremove200}
\end{table}

For the CP-VRP, based on the results shown in Table \ref{noremove100_cp}, we pick the best five ranges as follows: (10\%-35\%), (10\%-40\%), (10\%-45\%), (10\%-50\%), and (5\%-50\%) to continue tuning the 200 customer instances. Table \ref{noremove200_cp} shows the average results for all 25, 50, 100, and 200 customer instances. We can see that (5\%-50\%) is the best range for removing customers for the CP-VRP.

\begin{table}[h]
\centering
\small
\begin{tabular}{llllll}
\hline
 \diagbox{$\lambda_1$/\%}{$\lambda_2$/\%}  & 30     & 35 & 40 & 45     & 50     \\
   \hline
5  & 2.29\% & 2.26\% & 2.10\% & 2.14\% & \textbf{2.01\%} \\
10 & 2.35\% & 2.06\% & 2.05\% & 2.06\% & 2.03\% \\
15 & 2.34\% & 2.27\% & 2.18\% & 2.15\% & 2.28\% \\
20 & 2.32\% & 2.40\% & 2.34\% & 2.13\% & 2.35\% \\
25 & 2.34\% & 2.15\% & 2.31\% & 2.26\% & 2.11\%
\\
\hline
\end{tabular}
\caption{Tuning of range of customers to remove (Solomon 25, 50, 100) (CP-VRP)} \label{noremove100_cp}
\end{table}

\begin{table}[h]
\centering
\tiny
\begin{tabular}{cccccc}
\hline
Range & (10\%-35\%) & (10\%-40\%) & (10\%-45\%)&  (10\%-50\%) &  \textbf{(5\%-50\%)}    \\
\hline
Avg Gap/\%             & 3.17 & 2.98 & 2.78 & 2.71 & \textbf{2.57}\\
\hline
\end{tabular}
\caption{Tuning of range of customers to remove (25, 50, 100, 200) (CP-VRP)} \label{noremove200_cp}
\end{table}

\subsection{Comparison between the ALNS and the CPLEX on Small Instances}

We test the MIP models for both problems on the same data set, and we give a time limit of 3600s (1h) for CPLEX to solve each instance. The instance is set with three machines in each vehicle and a production coefficient $\mu = 2$. We generate the small-size instances by randomly picking 10, 15, and 20 customers from the Solomon 25 instances.

Table \ref{tb:cplex_mop} - \ref{tb:cplex_cp} compare the results obtained by CPLEX and ALNS. The first three columns present the number of customers, the instance names and numbers of the different instances. For the CPLEX, we present the average upper bound (UB), the average lower bound (LB), the average computational time, the LB GAP (calculated as $\frac{UB - LB}{LB}\times100\%$), the number of instances solved to optimality, and the number of instances where feasible solutions are found. For the ALNS, we present the average results, best results, and the average computational times. The last column shows the gap between the average results by ALNS and the UB by CPLEX, calculated as $\frac{Avg - UB}{UB}\times100\%$. As we present the results with one digit decimal, some cells with small values in the columns LB GAP and GAP appear as 0.0\%. 

\begin{table}[h]
\centering
\resizebox{\textwidth}{20mm}{
\begin{tabular}{ccc|ccccccc|cccc|c}
\hline
            &      &              &  &       &        & CPLEX        &           &             &  &   &ALNS       &          &  & GAP    \\
\hline
\# customer & Name & \# Instance & UB    & LB    & Time(s)  & LB GAP  & \# To Opt & \# Feasible &  & Avg   & Best  & Avg Time(s) &  &        \\
\hline
            & C    & 17           & 106.6 & 106.6 & 0.5    & 0.0\%   & 17        & 17          &  & 106.6 & 106.6 & 0.5      &  & 0.0\%  \\
10          & R    & 23           & 221.3 & 221.3 & 2.7    & 0.0\%   & 23        & 23          &  & 221.4 & 221.4 & 0.4      &  & 0.0\%  \\
            & RC   & 16           & 203.3 & 198.4 & 448.6  & 3.8\%   & 15        & 16          &  & 203.3 & 203.3 & 0.4      &  & 0.0\%  \\
            \hline
            & C    & 17           & 166.0 & 166.0 & 220.4  & 0.0\%   & 17        & 17          &  & 166.0 & 166.0 & 1.0      &  & 0.0\%  \\
15          & R    & 23           & 315.8 & 295.7 & 1724.7 & 6.2\%   & 14        & 23          &  & 314.3 & 314.3 & 1.0      &  & -0.5\%  \\
            & RC   & 16           & 288.3 & 177.8 & 2422.9 & 96.7\%  & 4         & 12          &  & 270.0 & 269.5 & 1.1      &  & -6.3\%  \\
            \hline
            & C    & 17           & 334.2 & 224.3 & 1710.8 & 43.5\%  & 9         & 17          &  & 305.6 & 305.6 & 2.2      &  & -8.6\%  \\
20          & R    & 23           & 357.0 & 333.3 & 1270.7 & 8.5\%   & 10        & 12          &  & 340.4 & 340.4 & 2.8      &  & -4.7\%  \\
            & RC   & 16           & 573.8 & 257.0 & 3600.0 & 130.1\% & 0         & 8           &  & 379.8 & 379.2 & 3.2      &  & -33.8\%\\
            \hline
\end{tabular}}
\caption{CPLEX v.s. ALNS for MoP-VRP}
\label{tb:cplex_mop}
\end{table}

As shown in Table \ref{tb:cplex_mop}, the proposed ALNS for the MoP-VRP finds the same or better quality solutions than CPLEX in most instances within a much shorter computational time. There is only 1 out of 168 instances where ALNS obtains a worse average result than that of CPLEX. The gap between the results by CPLEX and ALNS grows as the instance size increases. The C instances are easier to solve compared to the R and RC instances; it becomes increasingly difficult for the CPLEX to find a feasible solution within 1 hour as the instance size increases.

The proposed ALNS for the CP-VRP also outperforms CPLEX as can be seen from Table \ref{tb:cplex_cp}. Only 6 of 168 instances in total are worse than the CPLEX. CPLEX could not find any good solutions for the CP-VRP when there are 20 customers. 

\begin{table}[h]
\centering
\resizebox{\textwidth}{20mm}{
\begin{tabular}{ccc|ccccccc|cccc|c}
\hline
            &      &              &  &       &        & CPLEX        &           &             &  &   &ALNS       &          &  & GAP    \\
\hline
\# customer & Name & \# Instance & UB    & LB    & Time(s)   & LB GAP  & \# To Opt & \# Feasible &  & Avg   & Best  & Avg Time(s) &  &        \\
\hline
                        & C    & 17           & 129.9 & 129.9 & 186.8  & 0.0\%  & 17        & 17          &  & 129.9 & 129.9 & 0.6      &  & 0.0\%   \\
10          & R    & 23           & 215.8 & 215.7 & 3.4    & 0.0\%  & 23        & 23          &  & 215.9 & 215.9 & 0.7      &  & 0.0\%   \\
            & RC   & 16           & 174.7 & 174.6 & 38.0   & 0.0\%  & 16        & 16          &  & 174.7 & 174.7 & 0.8      &  & 0.0\%   \\
            \hline
             & C    & 17           & 166.2 & 166.0 & 261.5  & 0.1\%  & 16        & 17          &  & 166.2 & 166.2 & 1.2      &  & 0.0\%   \\
15          & R    & 23           & 304.4 & 302.8 & 704.1  & 0.5\%  & 20        & 23          &  & 303.8 & 303.4 & 1.4      &  & -0.2\%  \\
            & RC   & 16           & 233.9 & 194.9 & 2386.5 & 29.1\% & 6         & 16          &  & 230.3 & 230.0 & 1.6      &  & -1.5\%  \\
            \hline
            & C    & 17           & 356.7 & 227.2 & 2543.5 & 39.1\% & 5         & 17          &  & 319.0 & 319.0 & 2.5      &  & -10.6\% \\
20          & R    & 23           & 412.4 & 323.9 & 2031.0 & 27.6\% & 11        & 23          &  & 359.4 & 359.1 & 3.2      &  & -12.8\% \\
            & RC   & 16           & 418.5 & 244.4 & 3600.0 & 89.9\% & 0         & 15          &  & 382.5 & 382.2 & 4.0      &  & -8.6\% \\
            \hline
\end{tabular}}
\caption{CPLEX v.s. ALNS for CP-VRP}
\label{tb:cplex_cp}
\end{table}

\subsection{Comparison between MoP-VRP and CP-VRP on Large Benchmark Instances}

This section presents the results of the ALNS for the MoP-VRP and the CP-VRP on the Solomon instances and Gehring and Homberger's 200 customer instances. We analyse the characteristics of the solution and compare the two problems.

First, we need to determine a suitable value for the early production coefficient $\epsilon$ for the CP-VRP. \change{Values: 0, 0.25, 0.5, 0.75, and 1} are selected and tested on the Solomon 25 and 50 instances. \change{In the CP-VRP, vehicles need to first wait for the production then start delivery. To avoid infeasibility caused by the waiting at the depot, we set the vehicle duration 10 times of that in the MoP-VRP.} As can be seen from the table, \change{the travel costs in the MoP-VRP and the CP-VRP are similar, but the delay costs are quite different. Comparing the delay in the CP-VRP without early production (i.e., $\epsilon = 0$) and in the MoP-VRP, we can see that mobile production helps reduce the total delay by more than 95\%. As $\epsilon$ increases, the delay in the CP-VRP decreases dramatically. When $\epsilon$ reaches 1, all the products are nearly ready for delivery at time 0, and the corresponding CP-VRP is close to a VRPTW. When $\epsilon = 0.75$, the delay in the CP-VRP is comparable to that in the MoP-VRP, we therefore set $\epsilon$ to 0.75 in the rest of the tests. }

\begin{table}[h]
\centering
\tiny
\change{
\begin{tabular}{ccccccccc}
\hline
                      n   &      & \multicolumn{3}{c}{25}            &  & \multicolumn{3}{c}{50}            \\
                         \hline
                       & Name & Avg travel & Avg delay & Avg cost &  & Avg travel & Avg delay & Avg cost \\
\hline
\multirow{3}{*}{$\epsilon = 0$}       & C    & 240.3      & 1513.3    & 1753.6   &  & 447.5      & 1945.1    & 2392.6   \\
                         & R    & 415.5      & 1507.7    & 1923.2   &  & 744.4      & 2270.0    & 3014.4   \\
                         & RC   & 394.2      & 4634.4    & 5028.6   &  & 764.1      & 6082.5    & 6846.6   \\
                         \hline
\multirow{3}{*}{$\epsilon = 0.25$}    & C    & 239.1      & 855.8     & 1094.9   &  & 428.8      & 747.7     & 1176.4   \\
                         & R    & 416.9      & 695.5     & 1112.4   &  & 762.3      & 731.9     & 1494.2   \\
                         & RC   & 411.2      & 2471.4    & 2882.6   &  & 814.1      & 2423.8    & 3237.9   \\
                         \hline
\multirow{3}{*}{$\epsilon = 0.5$}     & C    & 238.5      & 400.3     & 638.9    &  & 404.8      & 202.5     & 607.2    \\
                         & R    & 438.6      & 174.4     & 612.9    &  & 744.6      & 143.1     & 887.7    \\
                         & RC   & 384.9      & 915.8     & 1300.7   &  & 807.8      & 550.3     & 1358.1   \\
                         \hline
\multirow{3}{*}{$\epsilon = 0.75$}    & C    & 229.0      & 116.5     & 345.5    &  & 382.3      & 39.9      & 422.2    \\
                         & R    & 433.6      & 35.1      & 468.7    &  & 701.9      & 22.2      & 724.1    \\
                         & RC   & 368.0      & 126.6     & 494.6    &  & 699.5      & 96.7      & 796.2    \\
                         \hline
\multirow{3}{*}{$\epsilon = 1$}       & C    & 221.3      & 11.2      & 232.5    &  & 380.8      & 0.6       & 381.4    \\
                         & R    & 434.6      & 23.8      & 458.4    &  & 694.2      & 17.8      & 712.0    \\
                         & RC   & 376.6      & 17.2      & 393.8    &  & 657.1      & 40.2      & 697.3    \\
                         \hline
\multirow{3}{*}{MoP-VRP} & C    & 231.9      & 35.8      & 267.7    &  & 394.8      & 102.9     & 497.7    \\
                         & R    & 457.5      & 50.8      & 508.4    &  & 776.4      & 94.9      & 871.3    \\
                         & RC   & 529.0      & 90.3      & 619.3    &  & 972.7      & 136.0     & 1108.7 \\
                         \hline
\end{tabular}
}

\caption{The Average Cost with Different Early Production Coefficient}
\label{tb:earlyprodcoeff}
\end{table}

\begin{figure}[htbp]
\centering
\includegraphics[width=0.5\textwidth]{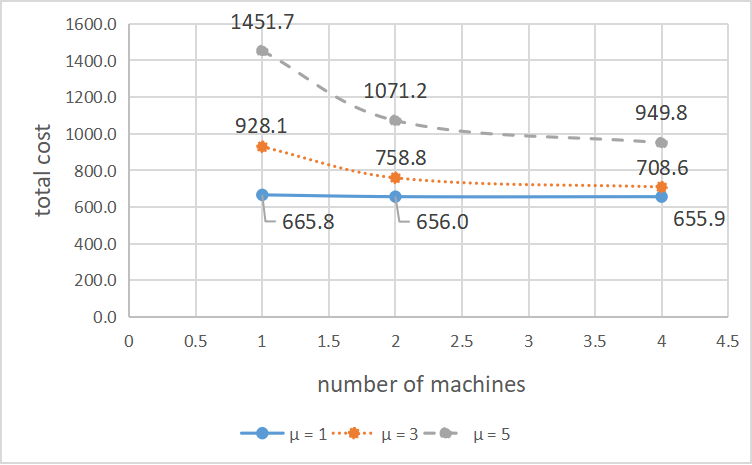}
\caption{Total cost v.s. $m/\kappa$ and $\mu$ \label{summary_bench}}
\end{figure}

We run the ALNS for the MoP-VRP and the CP-VRP on the Solomon instances and Gehring and Homberger’s 200 customer instances and find that the key factors that influence the objective value are the production time and the number of machines set in each vehicle. Figure \ref{summary_bench} summarizes how the total cost changes with the number of machines and production time for the MoP-VRP. As can be seen from the figure, when the machine number is fixed, the higher production time coefficient will lead to a higher total cost; when the production time coefficient ($\mu$) is fixed, the higher number of machines in each vehicle helps reduce the total cost. This trend becomes increasingly obvious as $\mu$ becomes larger and larger. 

When the production time coefficient is 1, the influence of the number of machines on the total cost is not obvious. This shows that when the production time is small, a few machines are enough, the number of machines is not the bottleneck to obtaining a good distribution plan for the MoP-VRP. When the production time coefficient increases (i.e., 3, 5), the total cost decreases significantly, typically when the number of machines in each vehicle increases from 1 to 2. The total cost decreases by a comparatively small amount when the number of machines in each vehicle increases from 2 to 4.

\subsection{Results on Realistic Instances}\label{sec:ResultsRealistic}

This section analyses the travel and delay costs for both the MoP-VRP and CP-VRP. The long term costs are then estimated, and advice on operation strategy is given.

\subsubsection{The delivery cost}

Table \ref{tb:real100} shows the average results for the MoP-VRP and CP-VRP on the realistic instances with 99 customers. The first three columns show scenario names, the number of machines in each vehicle, and the average number of vehicles. For both the MoP-VRP and the CP-VRP, we show the average traveled distance (in miles), the average delay (in minutes), the average total cost, and the average computational time for the MoP-VRP. In addition, we also show the average early production time for the CP-VRP. Each row in the table presents the average results of five random instances over ten random runs. 

\begin{table}[h]
\centering
\resizebox{\textwidth}{50mm}{
\begin{tabular}{cccccccccccccc}
\hline
                        &            &            &  & \multicolumn{4}{c}{MoP-VRP}       &  & \multicolumn{5}{c}{CP-VRP}                               \\
                        \cline{1-3} \cline{5-8} \cline{10-14}
                        & $m$ & $\kappa$&  & avg travel & avg delay & avg cost &avg time(s) &  & avg ahead & avg travel & avg delay & avg cost & avg time(s) \\
                        \cline{1-3} \cline{5-8} \cline{10-14}
\multirow{4}{*}{S\_W} & 1 & 5   &  & 319.9 & 1.6   & 321.5 & 47.9  &  & 366.3 & 321.0 & 0.8   & 321.8 & 101.4 \\
                      & 2 & 3.2 &  & 326.8 & 1.8   & 328.6 & 162.5 &  & 290.2 & 343.4 & 3.2   & 346.6 & 109.7 \\
                      & 3 & 3.2 &  & 324.7 & 1.8   & 326.5 & 244.7 &  & 193.5 & 338.3 & 2.4   & 340.7 & 133.5 \\
                      & 4 & 3.2 &  & 324.4 & 1.7   & 326.1 & 315.6 &  & 145.1 & 337.7 & 2.0   & 339.7 & 141.1 \\
                      \hline
\multirow{4}{*}{M\_W} & 1 & 6.8 &  & 359.3 & 9.6   & 368.9 & 31.8  &  & 382.2 & 318.5 & 0.7   & 319.2 & 215.1 \\
                      & 2 & 4   &  & 326.4 & 4.1   & 330.5 & 120.0 &  & 323.8 & 325.8 & 2.2   & 328.1 & 113.6 \\
                      & 3 & 3   &  & 338.5 & 3.9   & 342.5 & 247.9 &  & 287.8 & 342.9 & 3.1   & 345.9 & 114.4 \\
                      & 4 & 3   &  & 335.4 & 4.5   & 339.8 & 329.8 &  & 215.9 & 342.3 & 3.4   & 345.8 & 136.9 \\
                      \hline
\multirow{4}{*}{H\_W} & 1 & 8.2 &  & 456.2 & 143.5 & 599.7 & 23.7  &  & 403.3 & 317.4 & 1.7   & 319.1 & 290.1 \\
                      & 2 & 4.4 &  & 394.2 & 157.2 & 551.4 & 89.5  &  & 378.8 & 320.7 & 2.2   & 322.9 & 119.2 \\
                      & 3 & 3   &  & 369.4 & 17.0  & 386.5 & 196.0 &  & 366.9 & 345.6 & 4.1   & 349.7 & 118.3 \\
                      & 4 & 3   &  & 339.9 & 12.0  & 351.9 & 322.1 &  & 275.2 & 341.4 & 3.8   & 345.2 & 125.8 \\
                      \hline
\multirow{4}{*}{S\_T} & 1 & 5   &  & 373.1 & 111.6 & 484.6 & 40.3  &  & 367.9 & 335.9 & 2.2   & 338.1 & 100.2 \\
                      & 2 & 3   &  & 367.5 & 18.8  & 386.2 & 159.3 &  & 306.6 & 371.7 & 35.7  & 407.3 & 105.4 \\
                      & 3 & 2.6 &  & 367.1 & 163.0 & 530.0 & 304.8 &  & 245.6 & 363.7 & 268.1 & 631.8 & 132.4 \\
                      & 4 & 2.6 &  & 366.7 & 166.1 & 532.7 & 390.6 &  & 184.2 & 366.2 & 183.2 & 549.3 & 134.6 \\
                      \hline
\multirow{4}{*}{M\_T} & 1 & 7   &  & 383.1 & 32.8  & 415.8 & 30.2  &  & 368.5 & 331.8 & 1.5   & 333.3 & 189.8 \\
                      & 2 & 4   &  & 357.2 & 11.3  & 368.4 & 111.2 &  & 322.4 & 344.1 & 2.5   & 346.6 & 110.5 \\
                      & 3 & 3   &  & 373.8 & 16.4  & 390.2 & 236.6 &  & 286.6 & 376.1 & 13.5  & 389.6 & 114.8 \\
                      & 4 & 2.8 &  & 372.4 & 71.1  & 443.5 & 351.8 &  & 236.5 & 367.5 & 216.8 & 584.3 & 137.0 \\
                      \hline
\multirow{4}{*}{H\_T} & 1 & 8.4 &  & 432.2 & 196.0 & 628.2 & 23.4  &  & 397.8 & 330.3 & 1.6   & 331.9 & 287.7 \\
                      & 2 & 4.4 &  & 400.0 & 162.9 & 562.9 & 86.3  &  & 382.7 & 343.4 & 3.7   & 347.1 & 119.6 \\
                      & 3 & 3.2 &  & 387.0 & 161.5 & 548.6 & 189.3 &  & 352.0 & 372.7 & 20.4  & 393.1 & 110.3 \\
                      & 4 & 3   &  & 371.4 & 51.3  & 422.7 & 304.1 &  & 277.8 & 376.1 & 10.2  & 386.3 & 114.6        \\
                    \hline  
\end{tabular}    }  
\caption{Realistic instances - 99 customers}
\label{tb:real100}
\end{table}

\change{On a high level, the average cost for the MoP-VRP and CP-VRP is similar. Certain instance types are more attractive for either MoP-VRP or CP-VRP. At first, this may seem a bit disappointing from an MoP-VRP point of view, and the reader may ask what the advantage of MoP-VRP is. Here it is important to re-iterate that the early production coefficient $\epsilon$ was chosen to make the two problems perform similarly in terms of objective value and that the early production provides a huge advantage for the CP-VRP in terms of reducing delays (see the impact of the $\epsilon$ parameter in Table \ref{tb:earlyprodcoeff})}

The key advantage of mobile production is that it makes distribution more flexible. We find that in the CP-VRP, production has to start at least 2.5 hours (i.e., 145 mins) before the vehicles depart from the depot, and in some cases up to 6.5 (i.e. 400 mins) hours before the vehicles depart. Since the vehicles depart at 11 am, this means that almost no orders from the day of operations will be included in the CP-VRP. For the MoP-VRP, production starts when the vehicles leave the depot, which means it can accept orders until 11 am. Meanwhile, mobile production can keep an acceptable delay for customers. 

Regarding the influence of the number of machines in each vehicle, the overall finding is that setting one machine on each vehicle obtains the lowest total cost in the ``S\_W" scenario. Putting more than one machine in each vehicle lowers the total cost in most scenarios in the MoP-VRP. \change{We only consider up to four machines in each vehicle because this is what we estimate will be feasible to fit into the vans that we used to estimate long term costs.}

\subsubsection{The long-term operation cost estimation}

To put mobile production into practical use, not only should we estimate the optimal number of machines in each vehicle for good distribution but also we need to investigate how many machines in each vehicle will reduce long-term costs the most. Our long-term cost estimation is shown in Table \ref{tb:real100longterm}, and we summarize the results for long-term costs for 25, 50, and 99 customer instances in Table \ref{tb:sumlongterm}. \change{Considering that central production already exists in real-life practice, we only estimate the long-term cost of mobile production to check the feasibility of the new logistics mode. }

The prices we estimate in this work are based on a Danish setting, and costs related to 3D printing equipment are suggested by 3D Printhuset. All the estimated costs can be found in Table \ref{tb:costlist}.

\begin{table}[h]
\centering
\small
\begin{tabular}{lccclcc}
\hline
\textbf{Driver}            &  &                               &  & \textbf{Printer}           &  &                              \\
Hourly wage                &  & \euro{25.6}  &  & Cost of purchasing printer &  & \euro{2300} \\
Work hours per day         &  & 10 &   & Yearly maintenance cost    &  & \euro{345} \\
Work days per week         &  & 5                             &  & Renewal cycle              &  & 2yr            \\
Work weeks per year        &  & 50                            &  & Printer scrap value        &  & \euro{1000}                \\
\textbf{Vehicle}           &  &                               &   & Renewal cost               &  & \euro{1300} \\
Cost of purchasing vehicle &  & \euro{54000} &  & Renewal cost per year      &  & \euro{650} \\
Maintenance cost per year  &  & \euro{5400}  &  & Yearly cost                &  & \euro{995}   \\
\textbf{Fuel}              &  &                               &  &   \textbf{Planning Horizon}                         &  & 10yr   \\
Consumption                &  & 8.08mile/L                    &  &                       &  &                              \\
Price                      &  & \euro{1.1}/L &  &                            &  &     \\
\hline
\end{tabular}
\caption{Cost List}
\label{tb:costlist}
\end{table}

Regarding the wage for drivers, we assume that each worker works for 10 hour per day, five hours per week and 50 weeks per year. The average hourly wage is set to be \euro 25.6, which is the hourly wage for an educated driver in Denmark plus a 30 percent addition that covers insurance, holiday pay, and indirect costs. This means that the cost per driver a year is \euro 63,908. If workers without taking driver education are used for the job, the wage costs will be lower. \change{If we want to estimate the long-term costs associated with the central production mode, we need to consider that the wages for workers will be different because we need more workers in the central warehouse to handle early production; if we let the drivers handle this work, we will need to pay them more. The vehicles used in the central production mode will be a lighter type because there is no need to install printers in them, therefore the price for vehicles will be different as well. }

We consider a 10-year planning horizon and the total cost for planning horizon equals the investment cost plus ten times the yearly cost. From the 10-year total cost, we can calculate the average cost per year and the average cost per order given the total number of orders per year.  
   
Table \ref{tb:real100longterm} shows how we calculate long-term cost. The first four columns present the scenario names, setting, and average travel distance. The next two columns are the number of vehicles and the number of machines the company should purchase. The columns under ``Investment Cost" show the vehicle cost, machine cost, and the total cost. The columns under ``Yearly costs" show the maintenance costs of vehicles, maintenance costs of printers, wages of drivers, fuel costs, and total cost each year. Column 15 shows the total cost over the 10-year planning horizon. Column 16 shows the ratio of the total costs to that when there is only one vehicle in the same scenario. The last three columns show the average cost per year, the number of orders per year, and the cost per order.

Table \ref{tb:sumlongterm} summarizes the proportion of each cost component in the total cost for the instances in different sizes and with different numbers of machines per vehicle. The third column presents the ratio of the total costs to that when there is only one vehicle for each data size. We can see that the drivers' salary takes up the largest part of the total cost (more than 60\%) and that adding more machines to each vehicle can reduce the overall cost.  

Figure \ref{costperdelivery} shows how the average cost per order changes with the number of vehicles at each stage. Figure \ref{costperdelivery} illustrates that when the number of machines in each vehicle is fixed, the average cost per order gradually decreases even as the total number of customers increases. When the number of customers is fixed, the average cost per order drops dramatically when the number of machines increases from 1 to 2. Based on these findings, it is estimated that the cost for mobile production and distribution will continue to decrease as the business mode grows in popularity and technology matures.

\begin{table}[h]
\centering
\resizebox{\textwidth}{40mm}{
\begin{tabular}{ccccccccccccccccccccccc}
\hline
 &            &            &            &  & Vehicles & Machines & \multicolumn{3}{c}{Investment cost} &  & \multicolumn{5}{c}{Yearly costs (Maintenance \& Wage \& Fuel)}                   &  & total cost   & compared to   &  & Cost     & Orders & Cost per \\
 \cline{8-10} \cline{12-16}
Scenario                    & $m$ & $\kappa$ & avg travel &  & to buy   & to buy   & vehicle   & printer   & total      &  & vehicle & printer & drivers  & Fuel   & total    &  & over 10-year & 1 machine(\%) &  & per year & per year   & order \\
\cline{1-4} \cline{6-10} \cline{12-16} \cline{18-19} \cline{21-23}
\multirow{4}{*}{S\_W} & 1 & 5   & 319.9 &  & 6  & 6  & 324000 & 13800 & 337800 &  & 32400 & 5970  & 319540 & 10911 & 368821 &  & 4026008 & 100.0 &  & 402601 & 24750 & 16.3 \\
                      & 2 & 3.2 & 326.8 &  & 5  & 10 & 270000 & 23000 & 293000 &  & 27000 & 9950  & 204505 & 11149 & 252604 &  & 2819042 & 70.0  &  & 281904 & 24750 & 11.4 \\
                      & 3 & 3.2 & 324.7 &  & 5  & 15 & 270000 & 34500 & 304500 &  & 27000 & 14925 & 204505 & 11076 & 257507 &  & 2879567 & 71.5  &  & 287957 & 24750 & 11.6 \\
                      & 4 & 3.2 & 324.4 &  & 5  & 20 & 270000 & 46000 & 316000 &  & 27000 & 19900 & 204505 & 11067 & 262472 &  & 2940723 & 73.0  &  & 294072 & 24750 & 11.9 \\
                      \hline
\multirow{4}{*}{M\_W} & 1 & 6.8 & 359.3 &  & 8  & 8  & 432000 & 18400 & 450400 &  & 43200 & 7960  & 434574 & 12257 & 497991 &  & 5430306 & 100.0 &  & 543031 & 24750 & 21.9 \\
                      & 2 & 4   & 326.4 &  & 5  & 10 & 270000 & 23000 & 293000 &  & 27000 & 9950  & 255632 & 11133 & 303715 &  & 3330150 & 61.3  &  & 333015 & 24750 & 13.5 \\
                      & 3 & 3   & 338.5 &  & 4  & 12 & 216000 & 27600 & 243600 &  & 21600 & 11940 & 191724 & 11547 & 236811 &  & 2611708  & 48.1  &  & 261171 & 24750 & 10.6 \\
                      & 4 & 3   & 335.4 &  & 4  & 16 & 216000 & 36800 & 252800 &  & 21600 & 15920 & 191724 & 11440 & 240683 &  & 2659634 & 49.0  &  & 265963 & 24750 & 10.7 \\
                      \hline
\multirow{4}{*}{H\_W} & 1 & 8.2 & 456.2 &  & 10 & 10 & 540000 & 23000 & 563000 &  & 54000 & 9950  & 524045 & 15563 & 603558 &  & 6598576 & 100.0 &  & 659858 & 24750 & 26.7 \\
                      & 2 & 4.4 & 394.2 &  & 6  & 12 & 324000 & 27600 & 351600 &  & 32400 & 11940 & 281195 & 13446 & 338980 &  & 3741404 & 56.7  &  & 374140 & 24750 & 15.1 \\
                      & 3 & 3   & 369.4 &  & 4  & 12 & 216000 & 27600 & 243600 &  & 21600 & 11940 & 191724 & 12602 & 237866 &  & 2622257 & 44.8  &  & 262226 & 24750 & 10.6 \\
                      & 4 & 3   & 339.9 &  & 4  & 16 & 216000 & 36800 & 252800 &  & 21600 & 15920 & 191724 & 11594 & 240838 &  & 2661180 & 40.3  &  & 266118 & 24750 & 10.8 \\
                      \hline
\multirow{4}{*}{S\_T} & 1 & 5   & 373.1 &  & 6  & 6  & 324000 & 13800 & 337800 &  & 32400 & 5970  & 319540 & 12725 & 370635 &  & 4044147  & 100.0 &  & 404415 & 24750 & 16.3 \\
                      & 2 & 3   & 367.5 &  & 4  & 8  & 216000 & 18400 & 234400 &  & 21600 & 7960  & 191724 & 12534 & 233818 &  & 2572577  & 63.6  &  & 257258 & 24750 & 10.4 \\
                      & 3 & 2.6 & 367.1 &  & 4  & 12 & 216000 & 27600 & 243600 &  & 21600 & 11940 & 166161 & 12521 & 212222 &  & 2365818 & 58.5  &  & 236582 & 24750 & 9.6  \\
                      & 4 & 2.6 & 366.7 &  & 4  & 16 & 216000 & 36800 & 252800 &  & 21600 & 15920 & 166161 & 12507 & 216188 &  & 2414678 & 59.7  &  & 241468 & 24750 & 9.8  \\
                      \hline
\multirow{4}{*}{M\_T} & 1 & 7   & 383.1 &  & 8  & 8  & 432000 & 18400 & 450400 &  & 43200 & 7960  & 447356 & 13066 & 511582 &  & 5566218 & 100.0 &  & 556622 & 24750 & 22.5 \\
                      & 2 & 4   & 357.2 &  & 5  & 10 & 270000 & 23000 & 293000 &  & 27000 & 9950  & 255632 & 12183 & 304765 &  & 3340645 & 60.0  &  & 334065 & 24750 & 13.5 \\
                      & 3 & 3   & 373.8 &  & 4  & 12 & 216000 & 27600 & 243600 &  & 21600 & 11940 & 191724 & 12749 & 238013 &  & 2623725 & 47.1  &  & 262373 & 24750 & 10.6 \\
                      & 4 & 2.8 & 372.4 &  & 4  & 16 & 216000 & 36800 & 252800 &  & 21600 & 15920 & 178942 & 12704 & 229166 &  & 2544457 & 45.7  &  & 254446 & 24750 & 10.3 \\
                      \hline
\multirow{4}{*}{H\_T} & 1 & 8.4 & 432.2 &  & 10 & 10 & 540000 & 23000 & 563000 &  & 54000 & 9950  & 536827 & 14743 & 615520 &  & 6718195 & 100.0 &  & 671820 & 24750 & 27.1 \\
                      & 2 & 4.4 & 400.0 &  & 6  & 12 & 324000 & 27600 & 351600 &  & 32400 & 11940 & 281195 & 13645 & 339180 &  & 3743398 & 55.7  &  & 374340 & 24750 & 15.1 \\
                      & 3 & 3.2 & 387.0 &  & 5  & 15 & 270000 & 34500 & 304500 &  & 27000 & 14925 & 204505 & 13202 & 259632 &  & 2900823 & 43.2  &  & 290082 & 24750 & 11.7 \\
                      & 4 & 3   & 371.4 &  & 4  & 16 & 216000 & 36800 & 252800 &  & 21600 & 15920 & 191724 & 12667 & 241911 &  & 2671911 & 39.8  &  & 267191 & 24750 & 10.8\\
                        \hline
\end{tabular}}  
\caption{Realistic instances - 99 customers - long term cost}
\label{tb:real100longterm}
\end{table}

\begin{table}[h]
\begin{minipage}{0.45\textwidth}
\centering
\tiny
\begin{tabular}{ccccccc}
\hline
& & Compared to & Vehicle  & Printer & Driver  & Fuel  \\
n  & $m$  & 1 machine (\%) & cost &  cost &  cost &  cost \\
\hline
\multirow{4}{*}{25} & 1                          & 100.0                      & 19.79\%           & 2.24\%            & 75.17\%          & 2.80\%         \\
                    & 2                          & 67.0                       & 23.64\%           & 5.36\%            & 66.46\%          & 4.54\%         \\
                    & 3                          & 66.8                       & 23.75\%           & 8.07\%            & 63.64\%          & 4.54\%         \\
                    & 4                          & 68.6                       & 23.12\%           & 10.48\%           & 61.97\%          & 4.43\%         \\
                    \hline
\multirow{4}{*}{50} & 1                          & 100.0                      & 17.37\%           & 1.97\%            & 78.15\%          & 2.51\%         \\
                    & 2                          & 66.8                       & 19.48\%           & 4.41\%            & 72.18\%          & 3.93\%         \\
                    & 3                          & 64.1                       & 19.52\%           & 6.63\%            & 69.77\%          & 4.08\%         \\
                    & 4                          & 65.5                       & 19.09\%           & 8.65\%            & 68.27\%          & 3.98\%         \\
                    \hline
\multirow{4}{*}{99} & 1                          & 100.0                      & 16.00\%           & 1.81\%            & 79.68\%          & 2.50\%         \\
                    & 2                          & 61.2                       & 17.16\%           & 3.89\%            & 75.10\%          & 3.84\%         \\
                    & 3                          & 52.2                       & 17.52\%           & 5.96\%            & 71.89\%          & 4.63\%         \\
                    & 4                          & 51.3                       & 16.98\%           & 7.70\%            & 70.76\%          & 4.56\%   \\
                    \hline
\end{tabular}

\caption{Summary for long term cost} \label{tb:sumlongterm}
\end{minipage}\hfill
\begin{minipage}{0.5\textwidth}
\centering
\includegraphics[width=0.85\textwidth]{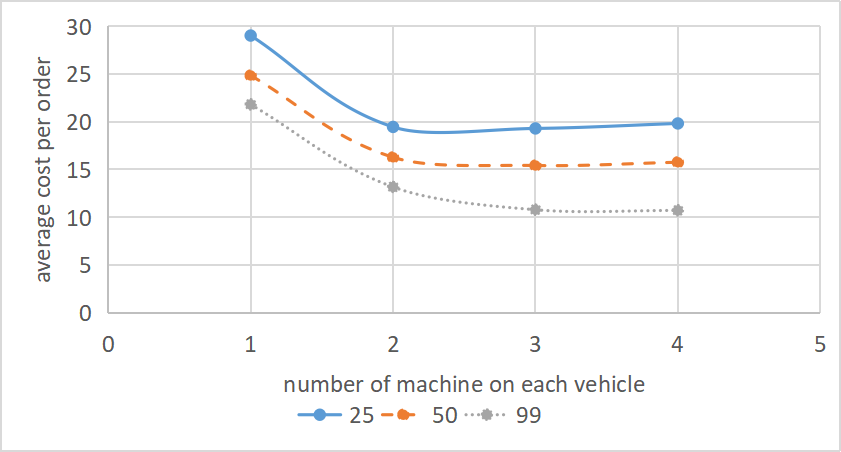}
\captionof{figure}{Cost Per Order \label{costperdelivery}}
\end{minipage}
\end{table}

\section{Conclusion}
This paper introduces a new variant of the vehicle routing problem, called the Mobile Production Vehicle Routing Problem (MoP-VRP), where the production takes place on the way to the customer. The MoP-VRP is highly complex, as it considers multiple products, multiple machines, soft time windows, and production schedules together within the distribution plan. For this complex problem, we propose a mixed integer programming (MIP) model as well as an adaptive large neighbourhood search (ALNS) heuristic. The computational results show that the proposed ALNS is highly efficient and can find the optimal solutions for most of the small instances within a short computational time. We also apply ALNS for the Central Production Vehicle Routing Problem (CP-VRP) and devise smart strategies to accelerate the solution process. One strategy is to avoid calculating redundant production sequences. The other is a piece-wise linear function, which speeds up the computation of delay costs when the production plan changes. The computational results show that the proposed algorithm for the CP-VRP is also efficient and solves most of the small instances to optimality within a short time. To investigate the advantage of the MoP-VRP, we generate realistic instances to compare with the CP-VRP. The instances are generated based on a Danish setting, and the values we use are estimated by 3D Printhuset. The experiments show that the key advantage of the MoP-VRP is flexibility: the MoP-VRP does not require early production and at the same time can keep a lower delivery cost than the CP-VRP. Regarding the realistic instances, we find that the MoP-VRP is feasible in practice due to its low operations costs from a long term estimation. The cost for each delivery will gradually decrease as the customer base increases.

We propose several directions for future research. \change{First, more constraints can be added to the basic model introduced in this paper to help improve the efficiency of this new logistics mode. For example, the concept of delayed differentiation can be introduced, as \citeauthor{su2010impact} (\citeyear{su2010impact}) show that it can result in shorter waiting times in MTO mode.} Second, exact algorithms can be developed for the MoP-VRP. The CPLEX can only solve the problem with up to 15 customers. It would be interesting to
develop an exact algorithm that can provide the optimal solution for larger size instances within a reasonable time. Last but not least, we recommend investigating the dynamic version of the MoP-VRP. It would be interesting to study how incoming requests might affect the final solution in the dynamic MoP-VRP, and how different waiting strategies can be used to improve the quality of the solution. 

\section*{Acknowledgement}
This work was supported by the National Science Foundation of China [11601436] and the Research Development Fund of Xi'an Jiaotong-Liverpool University [RDF-16-02-50].

\bibliography{references}
\section*{Appendix}

For both the \textbf{Geo Removal} and \textbf{Demand Removal}, we firstly remove a random seed customer. Then, the \textbf{Geo Removal} will remove \change{$\Phi-1$} customers that are closest to that seed customer. The \textbf{Demand Removal} will remove \change{$\Phi-1$} customers whose demands are closest to the seed customer. Alg. \ref{opt:geo} shows how the Geo Removal works. The Demand Removal shares the same framework. The only difference is that for the Geo Removal, \change{$c_{qi}$} in line 7 represents the distance between customers $q$ and $i$, whereas in Demand Removal we transfer \change{$c_{qi}$ }to $|d_q - d_i|$, which is the absolute value of the difference in demand between customers $q$ and $i$. 

\begin{algorithm}[H]
\begin{algorithmic}[1]
\State{\textbf{Function} Geo Removal \change{($s, u \in \mathbb{R}_+$)}}
\State{request $q$ = a randomly selected request from $s$;}
\State{Set of requests: \change{$R = \{q\}$};}
\While{$|R| < \Phi$}
\State{Array: $L$ = an array containing all request from $s$ not in $R$;}
\State{Sort $L$ such that}
\change{\State{$i < j \rightarrow c_{qL[i]} < c_{qL[j]}$;}}
\State{choose a random number \change{$\sigma$} from the interval [0,1);}
\State{\change{$R = R \cup \{L[\sigma^u|L|]\}$};}
\EndWhile
\State{remove the requests in $R$ from $s$;}
\end{algorithmic}
\caption{Geo Removal} \label{opt:geo}
\end{algorithm}

Like \citeauthor{ropke2006adaptive} (\citeyear{ropke2006adaptive}), a random parameter \change{$\sigma^u$} (appears in Alg. \ref{opt:geo} line 9 and Alg. \ref{opt:worst} line 6) is used to introduce some randomness to the worst removal and the four newly proposed removal operators, such that we can avoid removing the same customer over and over again. Parameter \change{$\sigma$} is a random number between 0 and 1 and \change{$u$} is a deterministic parameter. 

\change{The \textbf{worst removal} removes \change{$\Phi$} customers with the highest cost savings $\xi_{i}$, where $\xi_{i} = f(s) - f(\tilde{s})$, and $f(s)$ and $f(\tilde{s})$ are the cost of the solutions with customer $i$ and without $i$, respectively. The $f(s)$ and $f(\tilde{s})$ will be updated after we remove one customer from the solution, and the pseudo-code can be seen in Alg. \ref{opt:worst}. \textbf{Worst-Delay Removal} is to remove $\Phi$ customers with the highest delay cost savings and \textbf{Worst-Dist Removal} is to remove $\Phi$ customers with the highest travel distance savings, where the savings are calculated in the similar way as the \textbf{Worst Removal}}. 

\begin{algorithm}[H]
\begin{algorithmic}[1]
\State{\textbf{Function} Worst Removal (\change{$s , u \in \mathbb{R}_+$})}
\State{\change{$count = \Phi$};}
\While{\change{$count > 0$}}
\State{Array: $L$ = all planned requests $i$, sorted by descending \change{$\xi_{i}$};}
\State{choose a random number \change{$\sigma$} from the interval [0,1);}
\State{request: \change{$q = L[\sigma^u|L|]$};}
\State{remove the request $q$ from $s$;}
\State{\change{$count = count - 1$};}
\EndWhile
\end{algorithmic}
\caption{Worst Removal} \label{opt:worst}
\end{algorithm}

\end{document}